\documentclass[11pt,reqno,a4paper]{amsart}
\usepackage[a4paper,textwidth=158mm,textheight=244mm,centering]{geometry}

\usepackage[T1]{fontenc}
\usepackage{lmodern}
\usepackage{microtype}
\usepackage{mathtools}
\usepackage{amssymb}
\usepackage{mathrsfs}
\usepackage{graphicx}
\usepackage{enumitem}
\usepackage{xcolor}
\usepackage[colorlinks=true,linkcolor=blue!45!black,citecolor=blue!45!black,urlcolor=blue!45!black]{hyperref}

\hypersetup{pdftitle={Boundary Layers and Finite-Section Asymptotics for Small Polygons},
 pdfauthor={Dawid Trela},pdfsubject={Variational boundary layers and finite-section asymptotics}}

\numberwithin{equation}{section}
\theoremstyle{plain}
\newtheorem{theorem}{Theorem}[section]
\newtheorem{proposition}[theorem]{Proposition}
\newtheorem{lemma}[theorem]{Lemma}

\theoremstyle{definition}
\newtheorem{definition}[theorem]{Definition}
\newtheorem{hypothesis}[theorem]{Hypothesis}
\newtheorem{remark}[theorem]{Remark}

\newcommand{\R}{\mathbb R}
\newcommand{\N}{\mathbb N}
\newcommand{\J}{\mathcal J}
\newcommand{\Hh}{\mathcal H}
\newcommand{\Dnat}{D_N^{\natural}}

\newcommand{\eps}{\varepsilon}

\title[Boundary layers for small polygons]
{Boundary Layers and Finite-Section Asymptotics for Small Polygons}

\ifdefined\anonymoussubmission
  \author{}
\else
  \author{Dawid Trela}
  \address{Faculty of Law and Administration\\
War Studies University\\
al. gen. A. Chru\'sciela 103\\
00-910 Warsaw, Poland}
  \email{dawidmtrela@gmail.com}
  \thanks{Faculty of Law and Administration, War Studies University,
  al. gen. A. Chru\'sciela 103, 00-910 Warsaw, Poland.
  E-mail: dawidmtrela@gmail.com. ORCID: 0000-0001-9781-6425.}
\fi

\subjclass[2020]{Primary 39A12; Secondary 52A40, 37D10, 41A60}
\keywords{small polygon, isodiametric problem, boundary layer, stable manifold,
variational recurrence, asymptotic expansion}

\begin{document}

\begin{abstract}
In large even polygons of unit diameter, computed angle profiles oscillate near the vertex carrying the pendant diameter. We study the half-line variational problem behind this observation of Bingane and Mossinghoff. Its unique positive equilibrium approaches the regular state with alternating multiplier $(-3+\sqrt5)/2$. We determine the sharp energy errors for two ways of closing a finite section and enclose the limiting cubic coefficient in a rational interval of width below $9.4\times10^{-29}$. An explicit scaling identifies these energies with the cubic coefficients of the finite polygon constructions. Transferring the profile to unrestricted area maximizers is a further, conditional result: under the stated geometric hypothesis, the profile gives the $n^{-3}$ term of their maximal area. The recurrence, its finite-section estimates and the enclosure can be checked independently of that hypothesis.

\end{abstract}

\maketitle

\section{Introduction and main results}
\label{sec:introduction}

A regular polygon has essentially the same geometry at every vertex. An even polygon of largest area among those of diameter one has a distinguished vertex: one unit diameter attaches it to an odd cycle of the others \cite{FosterSzabo2007}. How far from that attachment does the disturbance remain visible? In the computations of Bingane and Mossinghoff, successive angles overshoot and undershoot their regular value. Adding more adjustable angles changes the area coefficient less and less \cite{BinganeMossinghoff2024}. The numerical behaviour suggests both an infinite profile and a question about the accuracy of its finite approximations.

The usual extremal question asks for the number $A_n$, the largest area of a small $n$-gon. The odd case is determined by Reinhardt's regular polygon. At even orders, exact low-order solutions and general upper and lower bounds give the background to the present problem \cite{Reinhardt1922,Graham1975,AudetEtAl2002,Mossinghoff2006,FosterSzabo2007,HenrionMessine2013,BinganeSequential2023,BinganeImproved2023}. Here we first ask a question which survives even if one sets aside the global maximization problem. What sequence of angular deviations minimizes the limiting cost of one marked end, and what error is introduced by fixing its far endpoint?

The answer comes from a cubic action. Its first boundary equation pushes the endpoint above the regular value when its neighbour is fixed at one. In the interior the linearized equation has a negative stable root, so the sign of a small deviation changes at each step. The argument below proves existence of the nonlinear profile, excludes any second positive profile, and shows that its stable amplitude is nonzero. Thus the oscillation seen in the calculations has a definite rate. The finite-section analysis also determines the constants multiplying the leading errors for the two terminal conditions.

For the area of an unrestricted maximizing polygon a geometric identification is needed. We state this input separately as Hypothesis~\ref{thm:imported-even}, following the companion preprint~\cite{TrelaGeneralEven2026}. The recurrence and its finite sections depend only on the formulas that follow.

Put
\begin{align}
 b(x)&=\frac{x^3}{3}-4x+\frac{11}{3},                                      \label{eq:b-def}\\
 \psi(x,y)&=\frac{(x+y)^3}{3}+\frac{2x^3}{3}-6x-4y+\frac{20}{3}.            \label{eq:psi-def}
\end{align}
The limiting marked orbit is the positive sequence $w^*=(w_j^*)_{j\ge0}$
satisfying
\begin{align}
 3(w_0^*)^2+(w_0^*+w_1^*)^2&=10,                                           \label{eq:intro-B}\\
 (w_{j-1}^*+w_j^*)^2+(w_j^*+w_{j+1}^*)^2+2(w_j^*)^2&=10
       \quad(j\ge1),                                                       \label{eq:intro-E}\\
 w_j^*&\longrightarrow1.                                                   \label{eq:intro-tail-state}
\end{align}

\begin{theorem}[Universal boundary layer]
\label{thm:intro-boundary-layer}
There is exactly one positive solution of
\eqref{eq:intro-B}--\eqref{eq:intro-tail-state}.  It is the unique minimizer
of
\begin{equation}
 \J(w)=b(w_0)+\sum_{j\ge0}\psi(w_j,w_{j+1})                                \label{eq:intro-J}
\end{equation}
on $\{w\colon w_j\ge0,\ w-1\in\ell^2(\N_0)\}$.  Moreover,
\begin{equation}
 w_j^*-1=c_*\lambda_s^j+O(|\lambda_s|^{2j}),\qquad
 \lambda_s=\frac{-3+\sqrt5}{2},\qquad c_*\ne0.                            \label{eq:intro-sharp-tail}
\end{equation}
\end{theorem}

The shortest pinned section already shows why the boundary is visible. With $w_1=1$, the equation for $w_0$ reduces to $4w_0^2+2w_0-9=0$ and gives $w_0=(\sqrt{37}-1)/4>1$. Deeper in the chain the first variation about the constant state is $h_{j-1}+3h_j+h_{j+1}=0$. The nonlinear theorem identifies which of these linear behaviours actually survives at infinity.

Thus the relaxation toward the regular bulk alternates.  The algebraic
number $\lambda_s$ is the stable eigenvalue of the limiting recurrence; its
reciprocal is the unstable eigenvalue.

For $n=2m$ set $N=n-1$.  Let $\eta_j^{(n)}$ denote the central angular
deficits along the marked half-cycle of the unique maximizing polygon.  Put
\begin{equation}
 M_*:=2\sum_{j\ge0}(w_j^*-1).                                               \label{eq:intro-Mstar}
\end{equation}

\begin{theorem}[Conditional shape theorem]
\label{thm:intro-shape}
Assume Hypothesis~\ref{thm:imported-even}. Uniformly for $0\le j\le m-1$,
\begin{equation}
 \eta_j^{(n)}=\frac{\pi w_j^*}{N+M_*}
 +O\left(n^{-3}+\frac{2^{-n/2}}{n}\right).                                \label{eq:intro-shape}
\end{equation}
Consequently,
\begin{equation}
 \eta_j^{(n)}-\frac{\pi}{N+M_*}
 =\frac{\pi c_*}{N+M_*}\lambda_s^j
 +O\left(\frac{|\lambda_s|^{2j}}n+n^{-3}+\frac{2^{-n/2}}n\right).         \label{eq:intro-shape-tail}
\end{equation}
\end{theorem}

Section~\ref{sec:finite-sections} connects the finite constructions of
Bingane and Mossinghoff to positive variational sections $\widehat q_r$.
An exact triangular transformation identifies their cubic coefficient
with $\J_R$. An excursion-clipping argument and a signed inverse comparison
then prove global minimization on the explicit inherited coefficient
polytope. Geometric realization of that minimizing coefficient vector is
stated separately; it requires the diameter and convexity constraints of
the reconstructed polygon.

\begin{theorem}[Sharp variational finite sections]
\label{thm:intro-finite-sections}
The two parity subsequences of $\widehat q_r$ decrease strictly to $q_*$.
As $r\to\infty$,
\begin{align}
 \widehat q_r-q_*
 &=\frac{\sqrt5}{8}c_*^2|\lambda_s|^{2r}(1+o(1))
 &&(r\ \text{even}),\label{eq:intro-rate-even}\\
 \widehat q_r-q_*
 &=\frac{-5+3\sqrt5}{16}c_*^2|\lambda_s|^{2r}(1+o(1))
 &&(r\ \text{odd}).\label{eq:intro-rate-odd}
\end{align}
In either parity class,
\begin{equation}
 \frac{\widehat q_{r+2}-q_*}{\widehat q_r-q_*}
 \longrightarrow |\lambda_s|^4.\label{eq:intro-ratio}
\end{equation}
\end{theorem}

The limiting coefficient has the exact definition
\begin{equation}
 \boxed{q_*:=\frac7{48}+\frac1{16}\J(w^*).}                               \label{eq:intro-qstar}
\end{equation}
A rational certificate gives $q_-<q_*<q_+$, where
\begin{equation}
\begin{aligned}
q_-&=0.115054983523324847448438781430160232707319569,\\
q_+&=0.115054983523324847448438781524054833260293086.
\end{aligned}                                                                         \label{eq:intro-qstar-box}
\end{equation}

\begin{theorem}[Conditional maximal-area asymptotics]
\label{thm:intro-area}
Assume Hypothesis~\ref{thm:imported-even}. As $n\to\infty$ through the even integers,
\begin{equation}
 \boxed{
 A_n=\frac\pi4-\frac{5\pi^3}{48n^2}
          -\frac{q_*\pi^3}{n^3}+O(n^{-4}).}                                \label{eq:intro-area}
\end{equation}
If
\begin{equation}
 \overline A_n=\frac n2\sin\frac\pi n
 -\frac{n-1}{2}\tan\frac{\pi}{2n-2}                                      \label{eq:intro-FS}
\end{equation}
is the Foster--Szab\'o upper bound, then
\begin{equation}
 \overline A_n-A_n=\left(q_*-\frac1{24}\right)\frac{\pi^3}{n^3}
 +O(n^{-4}).                                                               \label{eq:intro-gap}
\end{equation}
\end{theorem}

Figure~\ref{fig:boundary-layer} illustrates the profile and finite-section
convergence. The plotted values are numerical checks; the stated rates
follow from the endpoint analysis in Section~\ref{sec:parity-rate}.
\begin{figure}[tb]
\centering
\includegraphics[width=\textwidth]{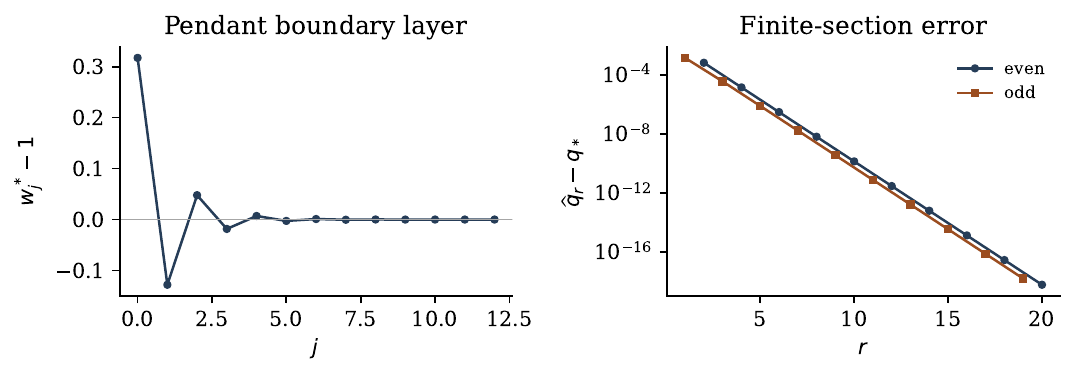}
\caption{The first thirteen deviations of the half-line orbit (left), and
finite-section energy errors for $1\le r\le20$ (right). The two terminal
conditions have different leading constants and the same two-step ratio.}
\label{fig:boundary-layer}
\end{figure}

Under the same geometric hypothesis, the area expansion persists to all orders.

\begin{theorem}[Analytic master expansion]
\label{thm:intro-master}
Assume Hypothesis~\ref{thm:imported-even}. There are $\eps>0$, $\vartheta\in(0,1)$, and a real-analytic function
\begin{equation}
 \mathscr A(t)=\frac\pi4+\sum_{k\ge2}a_kt^k\qquad(|t|<\eps)                \label{eq:intro-master-function}
\end{equation}
such that, for even $n$,
\begin{equation}
 A_n=\mathscr A(1/n)+O(\vartheta^n).                                       \label{eq:intro-master}
\end{equation}
The coefficients are recursively computable, and
$a_2=-5\pi^3/48$, $a_3=-q_*\pi^3$.
\end{theorem}

Sections~\ref{sec:global-input}--\ref{sec:scaled-limit} pass from the exact
finite equations to the limiting action.  Sections~\ref{sec:boundary-layer}
and~\ref{sec:hyperbolicity} prove Theorem~\ref{thm:intro-boundary-layer}.
Section~\ref{sec:shadowing-shape} returns to the true finite maximizers.
Sections~\ref{sec:finite-sections}--\ref{sec:parity-rate} establish the
finite-section laws.  Sections~\ref{sec:qstar}--\ref{sec:gap} derive the
sharp area coefficient and gap, and Section~\ref{sec:master-expansion}
proves the all-orders statement.

\section{Geometric input and its role}
\label{sec:global-input}

The companion work~\cite{TrelaGeneralEven2026} proposes a global
characterization of even-order maximizers. The present paper uses the
following precise package of geometric statements. We record it as a
hypothesis because its global reduction is not proved here; a proof of the
half-line recurrence alone would not establish it. In particular, the
fixed-$C$ uniqueness and reconstruction assertions below are needed to
identify a locally constructed critical branch with a polygon of globally
maximal area.

Let $n\ge8$ be even and set $N=n-1$.  The unit-distance skeleton of the
maximizer consists of an $N$-cycle and one pendant diameter.  Index the cycle
vertices by $i\in\mathbb Z/N\mathbb Z$, let $z_i\in\R^2$ be their vectors
from the Noether center, and put
\[
 u_i=|z_i|^2,\qquad r_i=\sqrt{u_i}.
\]
The Noether triangle $(0,z_i,z_{i+1})$ has unit base, central angle
$\delta_i$, and base angles $\alpha_i$ at $z_i$ and $\beta_i$ at
$z_{i+1}$.  Its central deficit is
\begin{equation}
 \eta_i:=\pi-\delta_i=\alpha_i+\beta_i.                                    \label{eq:def-eta}
\end{equation}
Write $\gamma_i$ for the ordinary skeleton angle at $z_i$ ($i\ne0$), and
write $2\theta$ for the marked angle.

\begin{hypothesis}[Global geometric input]
\label{thm:imported-even}
For every even $n\ge8$, there is exactly one congruence class, modulo
reflection, of maximum-area small $n$-gons.  Its skeleton is the
cycle-with-pendant skeleton described above; all Noether triangles are
positive, nondegenerate, and consistently oriented on the principal branch.

There is a scalar $C$ such that
\begin{equation}
 u_i+C=\frac12\cos\gamma_i\quad(i\ne0),
 \qquad u_0+C=\frac12\cos\theta,                                           \label{eq:noether-relations}
\end{equation}
and the exact critical equations are
\begin{align}
 \gamma_i&=\beta_{i-1}+\alpha_i &&(i\ne0),                                 \label{eq:matching-ordinary}\\
 2\theta&=\beta_{N-1}+\alpha_0,                                            \label{eq:matching-marked}\\
 \sum_{i=0}^{N-1}\eta_i&=\pi.                                             \label{eq:turning}
\end{align}
These data are the unique critical point of the $C^2$ generating action
$\Phi_N$ on its natural domain $\Dnat$, and
\begin{equation}
 D^2\Phi_N(X)\prec0\qquad(X\in\Dnat).                                     \label{eq:Phi-concave}
\end{equation}
In particular, for every fixed $C$, the restriction of $u\mapsto
\Phi_N(C,u)$ to a convex $C$-slice of $\Dnat$ has at most one critical
point.  Finally, the exact reconstruction formula is
\begin{equation}
 A_n=\frac12\sum_{i\in\mathbb Z/N\mathbb Z}z_i\mathbin\times z_{i-2}
       +\sin\theta(1-\cos\theta),                                         \label{eq:exact-reconstruction}
\end{equation}
with the positive orientation fixed below. The associated marked Reuleaux
body has diameter one and satisfies the polygonization identity
\eqref{eq:polygonization}.
\end{hypothesis}

For reference, the action can be written explicitly.  If
\begin{equation}
 \Delta(u,v)=\frac14\sqrt{2uv+2u+2v-u^2-v^2-1}                             \label{eq:triangle-area}
\end{equation}
is the area of the triangle with sides $\sqrt u,\sqrt v,1$, and its base
angles are $\alpha,\beta$, set
\begin{equation}
 E(u,v)=4\Delta(u,v)-2u\alpha-2v\beta.                                    \label{eq:E-def}
\end{equation}
For $x\in(0,\pi/2)$ determined by $\cos x=2(u+C)$, put
\begin{equation}
 V_C(u)=x\cos x-\sin x.                                                    \label{eq:V-def}
\end{equation}
Then
\begin{equation}
 \Phi_N(C,u)=\sum_{i=0}^{N-1}E(u_i,u_{i+1})
 +2V_C(u_0)+\sum_{i=1}^{N-1}V_C(u_i)-2\pi C.                              \label{eq:Phi-def}
\end{equation}
The natural principal domain is defined by
\begin{equation}
 u_i>0,\qquad 0<2(u_i+C)<1,\qquad
 |\sqrt{u_i}-\sqrt{u_{i+1}}|<1<\sqrt{u_i}+\sqrt{u_{i+1}}.                 \label{eq:natural-domain}
\end{equation}
This domain is convex: the squared triangle inequality is the strict
superlevel condition $2uv+2u+2v-u^2-v^2-1>0$ for a concave
quadratic, and the remaining inequalities are affine. The derivative identities
\begin{equation}
 E_u=-2\alpha,\qquad E_v=-2\beta,\qquad
 (V_C)_u=(V_C)_C=2x                                                       \label{eq:action-derivatives}
\end{equation}
give
\begin{align}
 \partial_{u_i}\Phi_N&=2(\gamma_i-\beta_{i-1}-\alpha_i) &&(i\ne0),       \label{eq:Phi-ui}\\
 \partial_{u_0}\Phi_N&=2(2\theta-\beta_{N-1}-\alpha_0),                 \label{eq:Phi-u0}\\
 \partial_C\Phi_N&=2\left(2\theta+\sum_{i=1}^{N-1}\gamma_i-\pi\right).
                                                                            \label{eq:Phi-C}
\end{align}

Uniqueness and invariance under reversal force reflection symmetry.  We
label the marked vertex by $0$.  Reflection acts on vertices by $i\mapsto-i$
and on an edge $j=(j,j+1)$ by
\begin{equation}
 j\longmapsto-j-1.                                                         \label{eq:edge-reflection}
\end{equation}
In particular $\alpha_0=\beta_{N-1}=\theta$.  The fixed-$C$ consequence of
\eqref{eq:Phi-concave}, rather than merely uniqueness of the full
$(C,u)$-critical point, is the precise global bridge used in
Section~\ref{sec:shadowing-shape}.

\section{Exact finite equations and the scaled limit}
\label{sec:scaled-limit}

Throughout the rest of the paper,
\begin{equation}
 n=2m,\qquad N=n-1=2m-1.                                                   \label{eq:nNm}
\end{equation}

\subsection{The exact radial recurrence}

Fix $C$.  Given consecutive radii $r_{i-1},r_i$ at an ordinary vertex, the
principal branch in Hypothesis~\ref{thm:imported-even} gives
\begin{align}
 \gamma_i&=\arccos\bigl(2(r_i^2+C)\bigr),                                  \label{eq:exact-gamma}\\
 \beta_{i-1}&=\arccos\left(\frac{r_i^2+1-r_{i-1}^2}{2r_i}\right),          \label{eq:exact-beta}\\
 \alpha_i&=\gamma_i-\beta_{i-1}.                                          \label{eq:exact-alpha}
\end{align}
The cosine rule then determines the next radius:
\begin{equation}
 r_{i+1}=\mathcal R_C(r_{i-1},r_i)
 :=\sqrt{r_i^2+1-2r_i\cos\alpha_i}.                                       \label{eq:exact-radial-map}
\end{equation}
At the marked end, $\cos\theta=2(r_0^2+C)$ and
$\alpha_0=\theta$, so
\begin{equation}
 r_1^2=r_0^2+1-4r_0(r_0^2+C).                                              \label{eq:exact-pendant}
\end{equation}
These are exact equations; no small-angle approximation has been used.

\subsection{The homogeneous state}

For a homogeneous bulk deficit $g\in(0,\pi/2)$, direct substitution gives
\begin{equation}
 r_g=\frac1{2\cos(g/2)},\qquad
 u_g=\frac1{4\cos^2(g/2)},\qquad
 C_g=\frac12\cos g-\frac1{4\cos^2(g/2)}.                                 \label{eq:homogeneous}
\end{equation}
In particular,
\begin{equation}
 C_g=\frac14-\frac{5g^2}{16}+\frac{g^4}{96}
       -\frac{5g^6}{2304}+O(g^8).                                         \label{eq:Cg-series}
\end{equation}
Thus each $C<1/4$ sufficiently close to $1/4$ has a unique representation
$C=C_g$ with $g>0$.

Linearizing the ordinary $u$-critical equations at the homogeneous state
gives
\begin{equation}
 \xi_{i+1}+a(g)\xi_i+\xi_{i-1}=0,
 \qquad a(g)=2\cos g+\sec^2(g/2).                                         \label{eq:finite-Jacobi}
\end{equation}
For example, at $u=u_g$ and $\Delta=\tfrac12u_g\sin g$,
\[
 E_{uv}=-\frac1{2\Delta},\qquad
 E_{uu}=\frac{2u_g-1}{4u_g\Delta},\qquad
 (V_C)_{uu}=-\frac4{\sin g};
\]
the ratio of diagonal to off-diagonal Hessian coefficients is
$u_g^{-1}-2+4u_g=a(g)$.  The exact finite-$g$ multipliers therefore solve
\begin{equation}
 \lambda^2+a(g)\lambda+1=0.                                               \label{eq:finite-multipliers}
\end{equation}

\subsection{The limiting recurrence}

Set $C=C_g$ and scale the central deficits by
\begin{equation}
 w_i=\frac{\eta_i}{g}.                                                      \label{eq:def-w}
\end{equation}
In a bounded scaled regime write
\begin{equation}
 u_i=\frac14+g^2U_i+O(g^4).                                                \label{eq:u-scaled}
\end{equation}
Expansion of the Noether and matching equations gives the unique marked
coefficient
\begin{equation}
 U_0=\frac{5-w_0^2}{16},                                                    \label{eq:U0}
\end{equation}
and, at every ordinary radial index,
\begin{equation}
 U_i=\frac{5-(w_{i-1}+w_i)^2}{16}.                                         \label{eq:Ui}
\end{equation}
The exact chord identity
\begin{equation}
 1=u_i+u_{i+1}+2\sqrt{u_iu_{i+1}}\cos(gw_i)                               \label{eq:chord-u}
\end{equation}
then yields, at order $g^2$,
\begin{equation}
 (w_{i-1}+w_i)^2+(w_i+w_{i+1})^2+2w_i^2=10,\qquad i\ge1.                 \label{eq:limiting-recurrence}
\end{equation}
At the pendant boundary one obtains
\begin{equation}
 3w_0^2+(w_0+w_1)^2=10.                                                     \label{eq:limiting-boundary}
\end{equation}
Positivity selects the analytic local branch
\begin{equation}
 w_{i+1}=F(w_{i-1},w_i)
 :=\sqrt{10-2w_i^2-(w_{i-1}+w_i)^2}-w_i.                                  \label{eq:F-map}
\end{equation}
At the regular fixed point $(1,1)$,
\begin{equation}
 DF(1,1)=(-1,-3),\qquad
 \lambda_{s,u}=\frac{-3\pm\sqrt5}{2}.                                    \label{eq:golden-roots}
\end{equation}
Reflection at the opposite edge of a finite half-cycle becomes
\begin{equation}
 (w_{m-2}+w_{m-1})^2+w_{m-1}^2=5.                                         \label{eq:limiting-terminal}
\end{equation}

\section{The half-line variational problem}
\label{sec:boundary-layer}

Let
\begin{equation}
 \Hh=\{w=(w_j)_{j\ge0}: w_j\ge0,\ w-1\in\ell^2(\N_0)\}.                  \label{eq:H-space}
\end{equation}
For $w\in\Hh$ define $\J$ by \eqref{eq:intro-J}.  The series is absolutely
convergent: $\psi(1,1)=0$, $\nabla\psi(1,1)=0$, and an $\ell^2$ sequence is
bounded, so the cubic Taylor remainder is summable.

Direct differentiation gives
\begin{align}
 \frac{d}{dw_0}\{b(w_0)+\psi(w_0,w_1)\}
 &=3w_0^2+(w_0+w_1)^2-10,                                                  \label{eq:J-Euler-boundary}\\
 \partial_2\psi(w_{j-1},w_j)+\partial_1\psi(w_j,w_{j+1})
 &=(w_{j-1}+w_j)^2+(w_j+w_{j+1})^2+2w_j^2-10.                             \label{eq:J-Euler-bulk}
\end{align}
Thus the Euler equations of $\J$ are precisely
\eqref{eq:limiting-boundary} and \eqref{eq:limiting-recurrence}.

The local Hessian has the exact quadratic form
\begin{equation}
 D^2\psi(x,y)[p,q]^2=2(x+y)(p+q)^2+4xp^2.                                 \label{eq:psi-hessian}
\end{equation}
It is positive semidefinite on the nonnegative quadrant and positive
definite when $x,y>0$.  Moreover $b''(x)=2x\ge0$.  The action is strictly convex on $\Hh$. Indeed, if two nonnegative
sequences differ at coordinate $j$, their open line segment has $w_j>0$
there, while its direction has $h_j\ne0$. The term $4w_jh_j^2$ in
\eqref{eq:psi-hessian} makes the second derivative along that segment
strictly positive.

For the quantitative argument fix
\begin{equation}
 K=\left[\frac12,\sqrt{\frac52}\right].                                   \label{eq:K}
\end{equation}
On $K^2$,
\begin{equation}
 D^2\psi(x,y)\succeq(3-\sqrt5)I\succ\frac34I,
 \qquad b''(x)\ge1.                                                       \label{eq:strong-convexity}
\end{equation}
Indeed the least local Hessian occurs at $x=y=1/2$, where its eigenvalues
are $3\pm\sqrt5$.  Taylor expansion from $(1,1)$ also gives
\begin{equation}
 \psi(x,y)\ge\frac38\bigl((x-1)^2+(y-1)^2\bigr),
 \qquad (x,y)\in K^2.                                                     \label{eq:psi-coercive}
\end{equation}

\begin{lemma}[Uniform bounds for a stable positive orbit]
\label{lem:stable-bounds}
Every positive solution of \eqref{eq:limiting-boundary} and
\eqref{eq:limiting-recurrence} that tends to one satisfies
\begin{equation}
 \frac{\sqrt{30}-\sqrt{10}}4\le w_j\le\sqrt{\frac52}\qquad(j\ge0).      \label{eq:stable-bounds}
\end{equation}
In particular, it lies in the interior of $K^{\N_0}$.
\end{lemma}

\begin{proof}
Any value above one has a finite index at which the maximum $M$ is attained.
At an interior maximum,
\[
 10=(w_{j-1}+M)^2+(M+w_{j+1})^2+2M^2\ge4M^2;
\]
the marked equation gives the same conclusion at $j=0$.  Hence
$M\le\sqrt{5/2}$.  If the minimum $a$ is below one, it is likewise attained
at a finite index.  Using the upper bound for both neighbors gives
\[
 10\le2\left(a+\sqrt{\frac52}\right)^2+2a^2.
\]
Solving the quadratic inequality yields the lower bound in
\eqref{eq:stable-bounds}; the marked endpoint only improves it.
\end{proof}

\begin{lemma}[Existence by finite variational sections]
\label{lem:halfline-existence}
There exists $w^*\in\Hh\cap K^{\N_0}$ satisfying the limiting Euler
equations and $w_j^*\to1$.
\end{lemma}

\begin{proof}
For $L\ge0$, impose $w_{L+1}=1$ and minimize
\begin{equation}
 \J_L(w_0,\ldots,w_L)=b(w_0)+\sum_{j=0}^{L}\psi(w_j,w_{j+1})              \label{eq:J-L}
\end{equation}
on $K^{L+1}$.  Strong convexity gives a unique minimizer.  It is interior.
For an ordinary coordinate the worst derivative on the lower face is
\[
 2\left(\frac12+\sqrt{\frac52}\right)^2+\frac12-10<0,
\]
and the worst derivative on the upper face is
\[
 2\left(\sqrt{\frac52}+\frac12\right)^2+5-10>0.
\]
The marked and Dirichlet endpoint derivatives satisfy the same strict
inward-pointing test.

Let $w^{(L)}$ be the minimizer and extend it by ones.  Since the regular
vector is admissible, $\J_L(w^{(L)})\le0$.  With
$B_K=-\min_{x\in K}b(x)$, \eqref{eq:psi-coercive} gives
\begin{equation}
 \sum_{j\ge0}(w_j^{(L)}-1)^2\le\frac{8B_K}{3},                             \label{eq:L2-uniform}
\end{equation}
uniformly in $L$.  Coordinatewise compactness and a diagonal extraction
give a limit $w^*\in K^{\N_0}$.  Fatou's lemma gives $w^*-1\in\ell^2$, hence
$w_j^*\to1$.  Every fixed Euler equation is eventually unaffected by the
artificial terminal value, and so the limit satisfies
\eqref{eq:limiting-boundary}--\eqref{eq:limiting-recurrence}.
\end{proof}

\section{Hyperbolicity, uniqueness, and the golden tail}
\label{sec:hyperbolicity}

Near $(1,1)$ the recurrence is generated by the analytic map
\begin{equation}
 T(x,y)=(y,F(x,y)),\qquad
 DT(1,1)=\begin{pmatrix}0&1\\-1&-3\end{pmatrix}.                          \label{eq:T-linearization}
\end{equation}
Its eigenvalues are $\lambda_s=(-3+\sqrt5)/2$ and
$\lambda_u=(-3-\sqrt5)/2=\lambda_s^{-1}$.

\begin{lemma}[A stable tail has finite energy]
\label{lem:stable-l2}
If a positive solution of the limiting equations satisfies $w_j\to1$, then
there are $C<\infty$ and $\rho<1$ such that
\begin{equation}
 |w_j-1|\le C\rho^j                                                     \label{eq:stable-geometric}
\end{equation}
for all sufficiently large $j$.  In particular, $w-1\in\ell^2(\N_0)$.
\end{lemma}

\begin{proof}
An orbit converging to the hyperbolic fixed point eventually belongs to its
one-dimensional local stable manifold.  The stable-manifold theorem applied
to \eqref{eq:T-linearization} gives geometric convergence
\cite[Chapter~6]{KatokHasselblatt1995}.
\end{proof}

\begin{theorem}[Unique stable pendant orbit]
\label{thm:stable-orbit}
The sequence $w^*$ of Lemma~\ref{lem:halfline-existence} is the unique
positive sequence satisfying \eqref{eq:limiting-boundary},
\eqref{eq:limiting-recurrence}, and $w_j\to1$.  It is the unique minimizer of
$\J$ on $\Hh$.
\end{theorem}

\begin{proof}
The first variation at $w^*$ vanishes on finitely supported directions.
The local Hessians are uniformly bounded on $K^2$, so the first-variation
operator is continuous on $\ell^2$; density extends the vanishing to every
$\ell^2$ direction.  Convexity therefore makes $w^*$ a global minimizer, and
strict convexity makes it unique.

Conversely, let $w$ be any positive stable orbit.  By
Lemma~\ref{lem:stable-bounds} it lies in $K$, and by
Lemma~\ref{lem:stable-l2} it belongs to $\Hh$.  Its Euler equations annihilate
the first variation first on finitely supported directions and then on all
of $\ell^2$.  It is consequently a critical point of the strictly convex
functional and must equal $w^*$.
\end{proof}

\begin{proposition}[Transversality and sharp asymptotic tail]
\label{prop:transversality}
The pendant curve \eqref{eq:limiting-boundary} intersects the stable manifold
of $(1,1)$ transversely at $w^*$.  There is a nonzero constant $c_*$ such
that
\begin{equation}
 w_j^*-1=c_*\lambda_s^j+O(|\lambda_s|^{2j}).                               \label{eq:sharp-tail}
\end{equation}
In particular, for every $\rho\in(|\lambda_s|,1)$,
$|w_j^*-1|\le C_\rho\rho^j$.
\end{proposition}

\begin{proof}
Tangency would produce, by differentiation, a nonzero decaying Jacobi field
satisfying the linearized pendant condition.  Such a field would be an
$\ell^2$ kernel vector of $D^2\J(w^*)$.  The coercivity inherited from
\eqref{eq:strong-convexity} gives
\begin{equation}
 D^2\J(w^*)[h,h]\ge\frac34\|h\|_2^2,                                     \label{eq:J-coercivity}
\end{equation}
so this is impossible.

The restriction of $T$ to its one-dimensional analytic stable manifold has
an analytic linearizing coordinate.  Expanding the inverse coordinate at
zero gives \eqref{eq:sharp-tail}.  If $c_*=0$, the stable coordinate of the
orbit would be zero and the orbit would be identically one.  This contradicts
$3\cdot1^2+(1+1)^2=7\ne10$ at the pendant boundary.
\end{proof}

\section{Uniform finite-cycle shadowing and the conditional shape theorem}
\label{sec:shadowing-shape}

The limiting orbit lives on a half-line. A polygon has a second end and the number of vertices is finite. To use the limiting orbit as a model one must control both the perturbation of the equations and the boundary condition at that distant end. This is the point where the geometric hypothesis is used.

Throughout this section we assume Hypothesis~\ref{thm:imported-even}.
The limiting orbit describes the maximizer after two uniform statements: the maximizing Noether data enter the common small-angle regime,
and the reflected finite boundary-value problem shadows $w^*$ without a
constant growing with $m$.

\subsection{Localization}

Let $P_n$ be the maximizer from Hypothesis~\ref{thm:imported-even}, and let
$R_n$ be its marked Reuleaux body.  The polygonization identity included in the geometric hypothesis is
\begin{equation}
 A(P_n)=A(R_n)-\left(2s(\theta)+\sum_{i\ne0}s(\gamma_i)\right),
 \qquad s(z)=\frac{z-\sin z}{2}.                                           \label{eq:polygonization}
\end{equation}
The isodiametric inequality gives $A(R_n)\le\pi/4$, whereas feasible regular
even polygons have areas tending to $\pi/4$.  Hence $A_n\to\pi/4$, the
nonnegative loss in \eqref{eq:polygonization} tends to zero, and
\begin{equation}
 \max(\theta,\gamma_1,\ldots,\gamma_{N-1})\longrightarrow0.                \label{eq:angles-small}
\end{equation}
The Noether relations and the strict radial triangle inequality give
$C<1/4$.  If $C\le1/4-\eps$ on a subsequence, then
\eqref{eq:angles-small} keeps all radii uniformly above $1/2$; every unit
chord then has a central deficit bounded below, contradicting
\eqref{eq:turning}.  Thus $C\uparrow1/4$.

Put
\begin{align}
 s_N&=\sqrt{1-4C},&
 x_i&=\frac{2\sin(\gamma_i/2)}{s_N}\quad(i\ne0),                           \notag\\
 x_0&=\frac{2\sin(\theta/2)}{s_N},&
 y_i&=\frac{2\sin(\eta_i/2)}{s_N}.                                        \label{eq:local-vars}
\end{align}
Then exactly
\begin{equation}
 u_i=\frac14\{1+s_N^2(1-x_i^2)\}.                                        \label{eq:local-u}
\end{equation}
Concavity of the square root and
$\sqrt{u_i}+\sqrt{u_{i+1}}>1$ imply $x_i^2+x_{i+1}^2<2$.  The chord equation
and triangle sine law now give, uniformly in $i$,
\begin{align}
 y_i^2&=4-2x_i^2-2x_{i+1}^2+O(s_N^2),                                    \label{eq:loc-chord}\\
 x_i&=\frac{y_{i-1}+y_i}{2}+O(s_N^2) &&(i\ne0),                            \label{eq:loc-split}\\
 x_0&=\frac{y_0}{2}+O(s_N^2).                                              \label{eq:loc-marked}
\end{align}

\begin{lemma}[Uniform nonvanishing]
\label{lem:y-nonzero}
There are $c>0$ and $n_0$ such that $y_i\ge c$ on every edge of every
maximizer with even $n\ge n_0$.
\end{lemma}

\begin{proof}
Suppose $y_{i_k}\to0$ and reflect so that $0\le i_k\le m_k-1$.  Take limits
in \eqref{eq:loc-chord}--\eqref{eq:loc-marked}.  If $i_k=0$, then $X_0=0$
and $X_1\le1$, while the edge-zero chord equation requires
$X_0^2+X_1^2=2$.  At an interior edge write the adjacent limiting deficits
as $Y_-,0,Y_+$.  Then $X_i=Y_-/2$, $X_{i+1}=Y_+/2$, and the vanishing-edge
equation forces $Y_-=Y_+=2$ and $X_i=X_{i+1}=1$; the preceding edge equation
with deficit two instead forces $X_i=0$.  At the fixed opposite edge,
reflection gives $X_m=X_{m-1}$; the vanishing-edge equation gives
$X_{m-1}=1$, then splitting gives $Y_{m-2}=2$, and the preceding chord
equation gives $X_{m-1}=0$.  All three cases are contradictory.
\end{proof}

The chord equation also gives a uniform upper bound for $y_i$.  Consequently
$c_1s_N\le\eta_i\le C_1s_N$.  Summing and using \eqref{eq:turning} yields
\begin{equation}
 s_N=\Theta(N^{-1}).                                                        \label{eq:sN-scale}
\end{equation}
By \eqref{eq:Cg-series}, the parameter $g$ determined by $C=C_g$ satisfies
\begin{equation}
 g=\Theta(N^{-1}).                                                          \label{eq:g-scale}
\end{equation}

\subsection{The reflected analytic system}

Set
\begin{equation}
 a_i=\frac{u_i-u_g}{g^2}.                                                   \label{eq:def-a}
\end{equation}
For ordinary rows define, with $u(a)=u_g+g^2a$,
\begin{align}
 P_g(a)&=\frac1g\arccos(\cos g+2g^2a),                                    \label{eq:Pg}\\
 A_g(a,b)&=\frac1g\arccos\left(\frac{u(a)+1-u(b)}{2\sqrt{u(a)}}\right), \label{eq:Ag}\\
 B_g(a,b)&=\frac1g\arccos\left(\frac{u(b)+1-u(a)}{2\sqrt{u(b)}}\right). \label{eq:Bg}
\end{align}
The half-cycle equations are
\begin{align}
 G_{g,0}(a)&=P_g(a_0)-A_g(a_0,a_1)=0,                                     \label{eq:G0g}\\
 G_{g,i}(a)&=P_g(a_i)-B_g(a_{i-1},a_i)-A_g(a_i,a_{i+1})=0                  \label{eq:Gig}
\end{align}
for $1\le i<m-1$, with the same terminal formula after setting
$a_m=a_{m-1}$.  The singularities at $g=0$ are removable; the functions are
real analytic in $g^2$ and the local variables on one common compact
neighborhood.

At $g=0$ put
\begin{equation}
 p_i=\sqrt{1-4a_i},\qquad
 w_i=\sqrt{1+8(a_i+a_{i+1})}.                                              \label{eq:pwa}
\end{equation}
Then
\begin{equation}
 G_{0,0}=p_0-\frac{w_0}{2},\qquad
 G_{0,i}=p_i-\frac{w_{i-1}+w_i}{2},                                       \label{eq:G-zero}
\end{equation}
which recovers the limiting pendant system.  The stable orbit corresponds
to
\begin{equation}
 a_0^*=\frac{4-(w_0^*)^2}{16},\qquad
 a_i^*=\frac{4-(w_{i-1}^*+w_i^*)^2}{16}\quad(i\ge1),                      \label{eq:a-star}
\end{equation}
and, because $|\lambda_s|<2/5$,
\begin{equation}
 |w_j^*-1|+|a_j^*|\le C(2/5)^j.                                           \label{eq:a-star-decay}
\end{equation}

\subsection{A dimension-independent inverse}

Fix
\begin{equation}
 \omega=\frac45,\qquad
 \|x\|_{1,\omega}^{(m)}=\sum_{j=0}^{m-1}\omega^{-j}|x_j|.                \label{eq:weighted-norm}
\end{equation}
Set
\begin{align*}
 \bar a^{(m)}&=(a_0^*,\ldots,a_{m-1}^*),&
 L_m&=D_aG_{0,m}(\bar a^{(m)}),\\
 \widehat w_{m-1}&=\sqrt{1+16a_{m-1}^*}.&&
\end{align*}
Direct differentiation, including the reflected endpoint, gives
\begin{equation}
 \begin{split}
 -h^TL_mh={}&\sum_{i=0}^{m-1}\frac2{p_i}h_i^2
 +\sum_{i=0}^{m-2}\frac2{w_i^*}(h_i+h_{i+1})^2
 +\frac4{\widehat w_{m-1}}h_{m-1}^2.
 \end{split}                                                               \label{eq:endpoint-coercivity}
\end{equation}
The orbit bounds imply, after increasing a fixed $m_0$,
\begin{equation}
 \sqrt{\frac85}I\preceq-L_m\preceq32I.                                   \label{eq:L-spectrum}
\end{equation}
Chebyshev approximation of $x^{-1}$ on this spectral interval and
tridiagonality give inverse decay as in \cite{DemkoMossSmith1984}: there are constants $C_D<\infty$ and $\sigma_0<3/4$
independent of $m$ such that
\begin{equation}
 |(L_m^{-1})_{ij}|\le C_D\sigma_0^{|i-j|}.                                 \label{eq:inverse-decay}
\end{equation}
Schur summation at the fixed weight $4/5$ yields
\begin{equation}
 \|L_m^{-1}f\|_{1,4/5}^{(m)}\le B_0\|f\|_{1,4/5}^{(m)},                  \label{eq:weighted-inverse}
\end{equation}
because $\sigma_0/(4/5)<1$ and $(4/5)\sigma_0<1$.

The homogeneous state $a_i=0$ solves every ordinary row for all $g$, so the
local analytic correction has the cancellation
\begin{equation}
 G_{g,i}(a)-G_{0,i}(a)
 =g^2H_i(g^2,a_{i-1},a_i,a_{i+1}),\qquad H_i(g^2,0,0,0)=0.                \label{eq:bulk-cancellation}
\end{equation}
Consequently the weighted bulk residual at $\bar a^{(m)}$ is $O(g^2)$,
not $O(mg^2)$.  The marked row contributes one $O(g^2)$ term and the
reflected terminal row contributes $O(2^{-m})$ in the weighted norm.  Thus
\begin{equation}
 \|G_{g,m}(\bar a^{(m)})\|_{1,4/5}^{(m)}\le C_R(g^2+2^{-m}).               \label{eq:finite-residual}
\end{equation}
Uniform local analyticity also gives
\begin{equation}
 \|D_aG_{g,m}(a)-L_m\|_{1,4/5\to1,4/5}
 \le C_1\bigl(g^2+\|a-\bar a^{(m)}\|_{1,4/5}^{(m)}\bigr).                \label{eq:finite-lipschitz}
\end{equation}
The frozen-inverse Newton map is therefore a contraction, uniformly in $m$,
and its unique zero satisfies
\begin{equation}
 \|a^{(m)}(g)-\bar a^{(m)}\|_{1,4/5}^{(m)}
 \le C(g^2+2^{-m}).                                                        \label{eq:a-shadow}
\end{equation}
The exact local map from adjacent $a$ variables to $w=\eta/g$ is analytic
and fixes the regular state.  Hence
\begin{equation}
 \boxed{\|w^{(m)}(g)-w^*\|_{1,4/5}^{(m)}
 \le C_S(g^2+2^{-m}).}                                                     \label{eq:w-shadow}
\end{equation}

We now identify this local branch with the true maximizer.  Reflecting the
zero $a^{(m)}(g)$ makes every full radial row
\eqref{eq:Phi-ui}--\eqref{eq:Phi-u0} vanish.  For small $g$, the common
positive-angle and triangle margins place the reflected point in
$\Dnat$.  The actual maximizer with the same $C=C_g$ belongs to the same
slice and solves the same radial equations.  Fixed-$C$ strict concavity in
Hypothesis~\ref{thm:imported-even} allows at most one such point.  The two
profiles therefore coincide.  This is the step at which the global geometric hypothesis identifies
the local branch with the maximizer.

Reflection and \eqref{eq:w-shadow} give
\begin{equation}
 \sum_{i=0}^{N-1}w_i^{(n)}=N+M_*+O(g^2+2^{-m}).                            \label{eq:w-sum}
\end{equation}
Since the exact turning identity is $g\sum_iw_i^{(n)}=\pi$ and
$g=\Theta(N^{-1})$,
\begin{equation}
 \boxed{g_n=\frac{\pi}{N+M_*}+O(N^{-4}).}                                 \label{eq:g-turning}
\end{equation}

\begin{proof}[Proof of Theorem~\ref{thm:intro-shape}]
The weighted bound implies, uniformly on the half-cycle,
\[
 w_j^{(n)}=w_j^*+O\bigl(n^{-2}+2^{-n/2}\bigr).
\]
Multiplication by \eqref{eq:g-turning} proves \eqref{eq:intro-shape}; then
\eqref{eq:sharp-tail} gives \eqref{eq:intro-shape-tail}.
\end{proof}

For completeness, the actual ordinary skeleton angles obey
\begin{equation}
 \gamma_j=\frac{g_n}{2}(w_{j-1}^{(n)}+w_j^{(n)})+O(g_n^3),\qquad
 \theta=\frac{g_n}{2}w_0^{(n)}+O(g_n^3).                                  \label{eq:skeleton-angle-expansion}
\end{equation}
Thus, for $j\ge1$,
\begin{equation}
 \gamma_j-\frac{\pi}{N+M_*}
 =\frac{\pi c_*}{2(N+M_*)}(\lambda_s^{j-1}+\lambda_s^j)
 +O\left(\frac{|\lambda_s|^{2j-2}}n+n^{-3}\right).                      \label{eq:skeleton-shape}
\end{equation}

\section{Finite sections and the inherited coefficient domain}
\label{sec:finite-sections}

For an even integer $R\ge2$, prescribe $w_R=1$ and set
\begin{equation}
 \J_R(w)=b(w_0)+\sum_{j=0}^{R-1}\psi(w_j,w_{j+1}).                         \label{eq:J-R}
\end{equation}
For $r\ge1$ put
\begin{equation}
 R(r)=2\left\lceil\frac r2\right\rceil.                                  \label{eq:R-r}
\end{equation}
Thus $R(r)=r$ for even $r$ and $R(r)=r+1$ for odd $r$.

Let $K$ be the interval in \eqref{eq:K}.  In the even case define
\begin{equation}
 \mathcal K_R^{\mathrm e}=\{w\in K^{R+1}:w_R=1\};                         \label{eq:K-even}
\end{equation}
in the odd case define
\begin{equation}
 \mathcal K_R^{\mathrm o}
 =\{w\in K^{R+1}:w_R=1,\ w_{R-2}+2w_{R-1}=3\}.                            \label{eq:K-odd}
\end{equation}
The Hessian formula \eqref{eq:psi-hessian} makes $\J_R$ uniformly strictly
convex on these compact convex sets, also after restriction to the odd
affine subspace.

\begin{definition}[Positive finite section]
\label{def:qhat}
Let $\widehat w^{(r)}$ be the unique minimizer of $\J_{R(r)}$ on
$\mathcal K_{R(r)}^{\mathrm e}$ when $r$ is even and on
$\mathcal K_{R(r)}^{\mathrm o}$ when $r$ is odd.  Define
\begin{equation}
 \widehat q_r=\frac7{48}+\frac1{16}\J_{R(r)}(\widehat w^{(r)}).            \label{eq:qhat}
\end{equation}
\end{definition}

\subsection{The finite geometric domain}

We reconstruct the leading domain of the finite construction in
\cite{BinganeMossinghoff2024}.  Write $R=2s$.  Its special half-chain angles
are
\begin{equation}
 \theta_0=\alpha,\qquad
 \theta_{2i-1}=\beta_i+\gamma_i,\qquad
 \theta_{2i}=\beta_i-\gamma_i\quad(1\le i\le s),                           \label{eq:BM-special-angles}
\end{equation}
followed by a constant tail angle $\beta$.  If
$\phi_R=\alpha+2\sum_{i=1}^s\beta_i$, the exact angle-sum and chord-closure
equations are
\begin{align}
 \phi_R+\left(\frac n2-R-1\right)\beta&=\frac\pi2,                         \label{eq:BM12}\\
 x_R+\frac{\sin(\phi_R-\beta/2)}{2\cos(\beta/2)}&=0.                       \label{eq:BM13}
\end{align}
The first equation eliminates the tail angle.  To describe the branch of
the second, put $\Phi_j=\sum_{k=0}^j\theta_k$,
\[
 X_{R-1}=\sum_{j=0}^{R-2}(-1)^j\sin\Phi_j,\qquad
 P=\alpha+2\sum_{i=1}^{s-1}\beta_i.
\]
The small-angle branch is
\begin{equation}
 \gamma_s=\arcsin\left(
 X_{R-1}+\frac{\sin(\phi_R-\beta/2)}{2\cos(\beta/2)}
 \right)-P-\beta_s.                                                       \label{eq:BM-gamma-branch}
\end{equation}
Its derivative in the closure equation is $-1+O(n^{-2})$; the other sine
branch is of order one.  For odd $r=R-1$, the same construction is used with
\begin{equation}
 \beta_s=\beta.                                                           \label{eq:BM-odd-convention}
\end{equation}

The written post-elimination free ranges are
\begin{equation}
 \frac{\pi}{2n-2}\le\alpha\le\frac\pi n,\qquad
 \frac\pi n\le\beta_i\le\frac{2\pi}n,\qquad
 0\le\gamma_i\le\frac\pi n,                                               \label{eq:BM-free-ranges}
\end{equation}
for the free variables.  In particular, the last $\gamma_s$ has already
been eliminated.  The proof below deliberately does not impose its former
interval after elimination.

Let $\mathcal D_{n,r}^{\rm geom}$ consist of exact points satisfying
\eqref{eq:BM12}--\eqref{eq:BM13}, the small-angle branch, the written free
ranges, and the geometric admissibility conditions: the vertices
reconstructed by the coordinate formulas of \cite[Section~3]{BinganeMossinghoff2024}
form a convex, consistently oriented polygon and every pair of vertices
has distance at most one. These last inequalities are part of the
definition, not a consequence of the two closure equations alone.
Define
\begin{equation}
 \mathcal D_r^{\rm geom}
 =\{\text{subsequential scaled limits of points in }
       \mathcal D_{n,r}^{\rm geom}\}.                                     \label{eq:Dgeom}
\end{equation}
With $\varepsilon=\pi/n$, write
\[
 \alpha=a\varepsilon+o(\varepsilon),\qquad
 \beta_i=B_i\varepsilon+o(\varepsilon),\qquad
 \gamma_i=c_i\varepsilon+o(\varepsilon).
\]
The first-order chord closure is
\begin{equation}
 a-2\sum_{i=1}^s c_i=\frac12.                                             \label{eq:first-closure}
\end{equation}
Thus the dependent coefficient is
\[
 c_s=\frac{a-1/2}{2}-\sum_{i=1}^{s-1}c_i.
\]
The outer projected domain $\mathcal P_r$ is given by
\begin{equation}
 \frac12\le a\le1,\qquad 1\le B_i\le2,\qquad
 0\le c_i\le1\quad(i<s),                                                  \label{eq:P-coeff}
\end{equation}
together with \eqref{eq:first-closure}; in the odd case $B_s=1$.  By
construction,
\begin{equation}
 \mathcal D_r^{\rm geom}\subseteq\mathcal P_r.                             \label{eq:D-in-P}
\end{equation}
If the eliminated $\gamma_s$ interval is retained, the true domain is
smaller; no step below depends on choosing between these interpretations.

\subsection{Triangular action identity}

Set
\begin{equation}
 t_0=a,\qquad t_{2i-1}=B_i+c_i,\qquad t_{2i}=B_i-c_i,                     \label{eq:t-BM}
\end{equation}
and introduce
\begin{equation}
 w_0=2t_0,\qquad w_j=2t_j-w_{j-1}\quad(1\le j\le R).                      \label{eq:BM-to-w}
\end{equation}
Equivalently,
\begin{equation}
 t_0=\frac{w_0}{2},\qquad t_j=\frac{w_{j-1}+w_j}{2}.                       \label{eq:w-to-BM}
\end{equation}
The alternating sum in \eqref{eq:first-closure} telescopes to $w_R=1$.
Writing
\begin{equation}
 e_i=w_{2i},\qquad o_i=w_{2i-1},                                          \label{eq:eo}
\end{equation}
one obtains
\begin{equation}
 e_0=2a,\qquad e_i=e_{i-1}-4c_i,\qquad
 o_i=2B_i-\frac{e_{i-1}+e_i}{2}.                                         \label{eq:eo-map}
\end{equation}
Consequently $\mathcal P_r$ is the compact affine polytope
\begin{equation}
 1\le e_0\le2,\quad e_s=1,\quad
 0\le e_{i-1}-e_i\le4\ (i<s),\quad 1\le B_i\le2,                          \label{eq:P-e}
\end{equation}
with $B_s=1$ for odd $r$.  There is intentionally no condition on the
terminal jump $e_{s-1}-e_s$.

\begin{proposition}[All-depth cubic identity]
\label{prop:all-depth-cubic}
For every even $R\ge2$, the cubic coefficient $Q_R$ obtained by expanding
the general construction equations satisfies
\begin{equation}
 Q_R(t(w))=\frac7{48}+\frac1{16}\J_R(w).                                  \label{eq:all-R}
\end{equation}
For odd $r=R-1$, condition \eqref{eq:BM-odd-convention} is exactly
$w_{R-2}+2w_{R-1}=3$.
\end{proposition}

\begin{proof}
Put $S=\sum_{k=0}^Rt_k$ and $T_j=\sum_{k=0}^jt_k$.  Expansion of the
eliminated closure variable gives
\[
 \kappa_R=\frac13\sum_{j=0}^{R-1}(-1)^jT_j^3
 +\frac16(S-\tfrac12)^3-\frac18(S-\tfrac12).
\]
Define
\begin{align*}
 F_R={}&-\frac{t_0^3}{6}
 -\frac16\sum_{k=2}^{R}\sum_{i=0}^{k-2}(-1)^i
 \left[\left(\sum_{j=0}^{i+1}t_{k-j}\right)^3
 -\left(\sum_{j=1}^{i+1}t_{k-j}\right)^3\right],\\
 H_R={}&\sum_{j=0}^{R-1}(-1)^j(S-T_j)^2.
\end{align*}
The general coordinate and area equations give
\[
 Q_R=\frac5{12}(R+1-S)
 -\left(F_R+\frac{5S}{24}-\frac1{48}+\frac14H_R+\frac12\kappa_R\right).
\]
After substitution of \eqref{eq:w-to-BM}, the $R=2$ identity is direct.  In
the induction $R\mapsto R+2$, let $a$ be the last retained $w$-coordinate
and insert $x,y$ before the new terminal value one.  Prefix cancellation
gives the following exact increment on the cubic side:
\[
 \begin{split}
 \Delta_{\rm BM}(a,x,y)=\frac1{48}\{&3a^2x-3a^2+3ax^2-3a+4x^3
      +3x^2y+3xy^2-30x\\
     &{}+4y^3+3y^2-27y+40\}.
 \end{split}
\]
Direct expansion gives the polynomial identity
\[
 \Delta_{\rm BM}(a,x,y)
 =\frac1{16}\{\psi(a,x)+\psi(x,y)+\psi(y,1)-\psi(a,1)\},
\]
which is precisely $1/16$ times the increment of \eqref{eq:J-R}.  The base
case and this generic polynomial identity prove the result.
\end{proof}

\subsection{Excursion clipping}

For a pair average $B$, define
\begin{equation}
 L_B(x,z)=\psi\left(x,2B-\frac{x+z}{2}\right)
 +\psi\left(2B-\frac{x+z}{2},z\right).                                   \label{eq:L-B}
\end{equation}
Then
\begin{equation}
 \J_R=b(e_0)+\sum_{i=1}^sL_{B_i}(e_{i-1},e_i).                            \label{eq:J-paired}
\end{equation}

\begin{lemma}[A subunit even excursion is improvable]
\label{lem:clipping}
If a point of $\mathcal P_r$ has $e_k<1$ for some $k<s$, another point of
$\mathcal P_r$ has strictly smaller action.
\end{lemma}

\begin{proof}
Put
\[
 P(u)=\frac7{12}u^3-\frac72u^2.
\]
Direct polynomial expansion gives, for the indicated ranges,
\begin{align*}
 \Delta_1&=L_B(1-u,1-v)-L_1(1-u,1-v),\\
 \Delta_2&=L_1(1-u,1-u-v)-\{P(u+v)-P(u)\},\\
 \Delta_3&=L_B(1+p,1-u)-P(u)-L_B(1+p,1).
\end{align*}
These differences satisfy
\begin{align}
 \Delta_1
 &=\frac{2q}{3}\{16q^2+6q(u+v)+36q                                      \notag\\
 &\hspace{42mm}{}+3u^2+6u+3v^2+6v\}\ge0,                                \label{eq:clip-1}\\
 \Delta_2
 &=\frac{u+v}{4}\{u(16-3v)+2v(10-v)\}\ge0,                               \label{eq:clip-2}\\
 \Delta_3
 &=\frac u4\{16B^2+8Bu-16B+p^2-pu                                      \notag\\
 &\hspace{42mm}{}+4p-2u^2+12u\}>0,                                     \label{eq:clip-3}\\
 L_1(1-U,1)&=-P(U),\qquad L_1(1,1)=0.                                    \label{eq:clip-4}
\end{align}
Here $q=B-1\ge0$ in \eqref{eq:clip-1}; in \eqref{eq:clip-2},
$u\ge0$ and $0\le v\le4$; and in \eqref{eq:clip-3},
$p,u\ge0$, $p+u\le4$, $B\ge1$, and $u>0$.  For the last sign, setting
$B=1+q$ turns the bracket into
\[
 16q(1+q)+8qu+2u(10-u)+p(p+4-u)>0.
\]

Let $k$ be the first index with $e_k<1$, and write
$e_{k-1}=1+p$, $e_i=1-u_i$ for $i\ge k$.  The $u_i$ are nondecreasing and
their increments are at most four.  Apply \eqref{eq:clip-3} to the crossing
edge, \eqref{eq:clip-1}--\eqref{eq:clip-2} to edges wholly below one, and
\eqref{eq:clip-1}, \eqref{eq:clip-4} to the terminal return to $e_s=1$.
The $P(u_i)$ terms telescope, giving
\[
 \sum_{i=k}^sL_{B_i}(e_{i-1},e_i)>L_{B_k}(e_{k-1},1).
\]
The right side is the action after setting $e_k=\cdots=e_s=1$ and
$B_{k+1}=\cdots=B_s=1$.  This replacement is feasible: its first new drop
is at most one, later drops vanish, and the odd condition $B_s=1$ is
preserved.
\end{proof}

Every global minimizer on $\mathcal P_r$ therefore satisfies
\begin{equation}
 1=e_s\le e_{s-1}\le\cdots\le e_0\le2.                                   \label{eq:clipped-evens}
\end{equation}
Equation \eqref{eq:eo-map} and $B_i\ge1$ then give $o_i\ge0$.  Thus
positivity is forced by an energy improvement; it has not been imposed by
definition.

On the clipped convex domain, for a variation $h_R=0$,
\begin{equation}
 \begin{split}
 D^2\J_R(w)[h,h]={}&2w_0h_0^2\\
 &+\sum_{j=0}^{R-1}
 \{2(w_j+w_{j+1})(h_j+h_{j+1})^2+4w_jh_j^2\}\\
 &\ge2\sum_{j=0}^{R-1}(h_j+h_{j+1})^2>0\qquad(h\ne0).
 \end{split}                                                              \label{eq:clipped-Hessian}
\end{equation}
The last strict inequality follows backward from $h_R=0$.  The same
argument applies on the odd affine hyperplane.

\subsection{Location of the finite minimizer}
For positive neighbors, let
\[
 T(x,z)=\frac{-(x+z)+\sqrt{40-3x^2+2xz-3z^2}}4
\]
be the positive root of the ordinary Euler equation. The next lemmas
justify the alternating order and all coefficient bounds used below.

\begin{lemma}[Interior finite minimizers]\label{lem:finite-interior}
Put $M=\sqrt{5/2}$ and $K=[1/2,M]$. The even and odd positive finite
sections have unique minimizers, and these minimizers lie in the relative
interior of their respective affine boxes. Consequently they satisfy all
unconstrained Euler equations and, in the odd case, the reduced terminal
Euler equation.
\end{lemma}
\begin{proof}
The Hessian is uniformly positive on $K$, including after restriction to
the odd affine subspace. For every ordinary coordinate the derivative at
the lower face is at most $2(1/2+M)^2+1/2-10<0$, and at the upper face it is
at least $2(M+1/2)^2+5-10>0$. The marked derivative has the same inward
signs. These signs exclude every active ordinary or marked box face.

For the odd problem write $x=w_{R-2}$ and $w_{R-1}=(3-x)/2$. For every
$x\in K$ the eliminated coordinate lies strictly inside $K$, since
$(3-M)/2>1/2$ and $5/4<M$. If $R\ge4$, the remaining terminal derivative is
\[
 \Phi(p,x)=\tfrac14(4p^2+8px+11x^2+14x-37),\qquad p=w_{R-3}.
\]
It is increasing in $p,x>0$. Uniformly for $p\in K$,
\[
 4\Phi(p,1/2)\le4M^2+4M-109/4<0,
 \qquad
 4\Phi(p,M)\ge18M-17/2>0.
\]
For $R=2$ the reduced derivative is
$(11x^2+14x-37)/4$, which has the same strict inward signs.
\end{proof}

\begin{lemma}[Signed inverse comparison]
Let $H$ be a symmetric positive definite tridiagonal matrix with strictly
positive adjacent off-diagonal entries. If $S=\operatorname{diag}((-1)^j)$,
then $(SHS)^{-1}$ has strictly positive entries.
\end{lemma}
\begin{proof}
The matrix $SHS$ is an irreducible symmetric positive definite matrix with
nonpositive off-diagonal entries. To see inverse positivity directly,
choose $a>\lambda_{\max}(SHS)$ so that $B=aI-SHS$ has nonnegative entries.
Its spectral radius is $a-\lambda_{\min}(SHS)<a$; hence
$(SHS)^{-1}=a^{-1}\sum_{k\ge0}(B/a)^k$. Irreducibility makes every entry
strictly positive in some term of this series.
\end{proof}

\begin{lemma}[Strict checkerboard order]\label{lem:checkerboard}
The finite minimizers satisfy
\[
 w_0>w_2>\cdots>w_{R-2}>w_R=1,
 \qquad w_1<w_3<\cdots<w_{R-1}<1.
\]
\end{lemma}
\begin{proof}
Use the independent coordinates $w_0,\ldots,w_{R-1}$ in the even case and
$w_0,\ldots,w_{R-2}$ in the odd case. At the regular vector, the gradient equals $-3\mathbf e_0$, where
$\mathbf e_0$ denotes the first coordinate vector. The mean Hessian along the segment from that vector to
the minimizer is symmetric positive definite, tridiagonal, and has
strictly positive adjacent off-diagonal entries. The previous lemma,
applied to $H(w-1)=3\mathbf e_0$, gives $(-1)^j(w_j-1)>0$ for every independent
coordinate. In the odd case the eliminated coordinate is also below one.

For the even problem with $R\ge4$, compare $w_j$ with $w_{j+2}$ on the
interval $0\le j\le R-2$, treating its last coordinate as Dirichlet data.
The two vectors satisfy the same ordinary equations in the interior. The
shifted vector has marked residual
\[
 3w_2^2+(w_2+w_3)^2-10=-w_1(2w_2+w_1)<0,
\]
where the ordinary equation at site two was used. Their Dirichlet
difference is $w_{R-2}-1>0$. Subtracting the two systems, integrating the
Hessian, and applying the alternating sign matrix produces a system whose
right side is nonnegative and nonzero: the left residual is positive, and
the right Dirichlet contribution is positive after changing the sign of
the last (odd-indexed) free row. Inverse positivity therefore gives
$(-1)^j(w_j-w_{j+2})>0$ for $0\le j\le R-3$. The case $R=2$ was already
covered by the alternating deviations from one.

For the odd problem with $R\ge4$, put
$p=w_{R-3}$, $x=w_{R-2}$, $y=w_{R-1}=(3-x)/2$, and $t=w_{R-4}$.
We already know $x>1$ and $p,y<1$. Write $P=4\Phi$. Then
\[
 P(y,x)=4(x-1)(2x+7)>0=P(p,x),
\]
so $p<y$. Let $z=T(p,p)$ be the positive root of
$2(p+z)^2+2z^2=10$. Because $p<1$, $z>1$, and
\[
 P(p,z)=(z-1)(3z+17)>0=P(p,x).
\]
Thus $x<z$, and hence
\[
 2(x+p)^2+2p^2-10
 <2(x+p)^2+2x^2-10<0.
\]
The ordinary Euler equation at $p$ now implies $t>x$. These are precisely
the last two strict same-parity comparisons. If $R\ge6$, repeat the
preceding signed-inverse comparison on $0\le j\le R-4$: the last
Dirichlet difference is $t-x>0$, the interior shifted equations are all
ordinary, and the shifted marked residual is again negative. If $R=4$
the two terminal comparisons already give the whole assertion. The case
$R=2$ follows from the alternating deviations alone.
\end{proof}

\begin{lemma}[Position inside the algebraic polytope]
The finite positive minimizer belongs to the relative interior of the
clipped polytope defined by the inherited free coefficient ranges. In
particular its free pair coefficients satisfy $1<B_i<2$ and $0<c_i<1$;
the only equality is $B_s=1$ in the odd case.
\end{lemma}
\begin{proof}
The strict checkerboard order gives $1<e_0<M<2$ and strictly decreasing
even coordinates ending at one. Therefore
$c_i=(e_{i-1}-e_i)/4\in(0,(M-1)/4)$, including the dependent terminal
coefficient. Consider an unconstrained odd site $o$ with adjacent even
coordinates $x,z\ge1$. Put $a=(x+z)/2$ and $d=(x-z)/2$; then
$1<a<M<2$ and $|d|\le a-1$. At the trial odd value $2-a>0$, the ordinary
Euler expression equals
\[
 2\{d^2+(2-a)^2-1\}
 \le2\{(a-1)^2+(2-a)^2-1\}<0.
\]
It increases strictly with the odd value, so $o>2-a$ and consequently
$B_i=(a+o)/2>1$. Also $B_i<M<2$ since all coordinates lie in $K$.
The odd terminal pair has $B_s=1$ by definition.
\end{proof}

\begin{theorem}[Minimization on the inherited coefficient polytope]
\label{thm:geometric-finite-sections}
For every $r\ge1$, the polynomial $Q_{R(r)}$ has a unique global minimizer
on $\mathcal P_r$. It is the positive finite minimizer
$\widehat w^{(r)}$, and
\[
 \min_{\mathcal P_r}Q_{R(r)}=\widehat q_r.
\]
\end{theorem}
\begin{proof}
The polytope is compact. Lemma~\ref{lem:clipping} places every minimizer
in the clipped nonnegative region. The finite positive minimizer belongs
to the relative interior of that region by the preceding lemmas and is a
critical point in the free coordinates. The Hessian in
\eqref{eq:clipped-Hessian} is positive definite there. Hence strict
convexity and the cubic identity \eqref{eq:all-R} prove the assertion.
\end{proof}

\subsection{What is needed for geometric attainment}
Define
\begin{equation}
 q_r^{\rm geom}:=\inf_{\mathcal D_r^{\rm geom}}Q_{R(r)}.
 \label{eq:qgeom-def}
\end{equation}
The inclusion \eqref{eq:D-in-P} gives the unconditional lower bound
$q_r^{\rm geom}\ge\widehat q_r$. If the minimizing coefficient vector is
realized by a sequence of admissible small polygons, then
\begin{equation}
 \widehat w^{(r)}\in\mathcal D_r^{\rm geom}
 \quad\Longrightarrow\quad q_r^{\rm geom}=\widehat q_r.
 \label{eq:qgeom-qhat}
\end{equation}
In that case uniqueness on $\mathcal P_r$ also gives uniqueness of the
limiting geometric minimizer.

The implicit-function theorem applied to \eqref{eq:BM12}--\eqref{eq:BM13}
provides a local exact solution of the angle-sum and chord-closure
equations with the minimizing leading coefficients. It does not, by
itself, check every pairwise diameter inequality of the reconstructed
polygon. Geometric attainment in \eqref{eq:qgeom-qhat} is therefore a
separate condition. All finite-section convergence statements in this
paper concern $\widehat q_r$ without requiring that condition.

\begin{remark}[The loose coefficient box]
\label{rem:loose-box}
The algebraic minimization theorem does not assert globality on the larger algebraic box used for
the later numerical minimization.  For $r=12$, the exact relaxed-box witness
\[
 a=0,\qquad B_1=\cdots=B_6=0,\qquad
 c_1=\cdots=c_5=\frac13,\qquad c_6=-\frac{23}{12}
\]
has $\J_{12}=-55/4$ and $Q_{12}=-137/192$.  It is not in
$\mathcal D_{12}^{\rm geom}$: the first inherited obstruction is
$a\ge1/2$, and it also violates every lower bound $B_i\ge1$.  Thus the
counterexample to loose-box globality and
Theorem~\ref{thm:geometric-finite-sections} concern different domains.
\end{remark}

\section{Sharp parity convergence of the finite sections}
\label{sec:parity-rate}

Fixing a distant value of the profile and releasing it lead to slightly different minimal energies. Their difference tends to zero, but its leading coefficient records which terminal condition was imposed. The alternating tail makes a comparison two sites apart especially informative.

Extend each finite minimizer by ones past its terminal site.  It is then an
admissible element of $\Hh$.  Since $w^*$ is the unique infinite minimizer,
\begin{equation}
 \widehat q_r>q_*\qquad(r<\infty).                                         \label{eq:qhat-above}
\end{equation}
Extending an $R$-section by two regular sites gives a competitor for the
$(r+2)$-section with the same action. In the even case the new ordinary
gradient at the former terminal site is $(w_{R-1}+1)^2-4<0$. In the odd
case the former eliminated coordinate $y=w_{R-1}<1$ becomes ordinary, and
its derivative is $4y(y-1)<0$, since the preceding coordinate is $3-2y$.
The extension is therefore not stationary in the larger affine box.
Strict convexity gives
\begin{equation}
 \widehat q_{r+2}<\widehat q_r.\label{eq:qhat-monotone}
\end{equation}

We determine the sharp error by resolving the remote terminal layer.  Let
$u^{(R)}$ denote the appropriate finite minimizer and put
$e^{(R)}=u^{(R)}-w^*$.  In the even problem, the comparison vector agrees
with $w^*$ through $R-1$ and replaces $w_R^*$ by one.  In the odd problem,
it agrees through $s_R=R-2$, imposes
$w_{s_R+1}=(3-w_{s_R})/2$, and appends the fixed value one.

All finite Hessians on the invariant compact box, including their
restrictions to the odd affine subspace, are uniformly coercive,
tridiagonal, and uniformly bounded.  Therefore
\begin{equation}
 |(H^{-1})_{ij}|\le C_0\sigma^{|i-j|}                                     \label{eq:finite-Demko}
\end{equation}
for fixed $\sigma<1$, uniformly in $R$ and for every Hessian on the segment
between the comparison vector and the minimizer.  Write $d^{(R)}=u^{(R)}-v^{(R)}$ in the independent finite coordinates;
this differs from $e^{(R)}=u^{(R)}-w^*$ at the imposed terminal sites.
If $G_R$ is the finite Euler map, the exact mean-value equation
\begin{equation}
 0=G_R(u^{(R)})=G_R(v^{(R)})+
 \left(\int_0^1DG_R(v^{(R)}+t(u^{(R)}-v^{(R)}))\,dt\right)d^{(R)}          \label{eq:finite-MV}
\end{equation}
then localizes the response at the same scale as the terminal residual.
The averaged Hessian retains the same uniform spectral bounds and
tridiagonal bandwidth, so the inverse-decay estimate applies to that
averaged matrix itself.

\subsection{Even terminal profile}

Let $r=R$ and $A_R=c_*\lambda_s^R$.  Replacing $w_R^*$ by one leaves only
the last Euler residual, which is
\begin{equation}
 (w_{R-1}^*+1)^2-(w_{R-1}^*+w_R^*)^2=-4A_R+o(A_R).                         \label{eq:even-forcing}
\end{equation}
For fixed $k\in\mathbb Z$ set
\begin{equation}
 E_R(k)=\frac{e_{R+k}^{(R)}}{A_R}.                                        \label{eq:E-even-def}
\end{equation}
Every local limit solves
$E_{k-1}+3E_k+E_{k+1}=0$ on the finite side, decays as
$k\to-\infty$, and has $E_0=-1$.  The unique full profile is
\begin{equation}
 E^{\mathrm e}_k=-\lambda_s^{|k|},\qquad k\in\mathbb Z.                   \label{eq:E-even}
\end{equation}

\subsection{Odd terminal profile}

Let $r=R-1$, $s_R=R-2=r-1$, and $A_R=c_*\lambda_s^{s_R}$.  After eliminating
$w_{s_R+1}=(3-w_{s_R})/2$, the reduced terminal first variation is
\begin{equation}
 \Phi(p,x)=\frac14(4p^2+8px+11x^2+14x-37).                               \label{eq:Phi-odd}
\end{equation}
At $(1,1)$, $\Phi_p=4$ and $\Phi_x=11$, while insertion of the stable tail
gives the normalized residual $4\lambda_s^{-1}+11$.  If
$D=\lim e_{s_R}^{(R)}/A_R$, the finite-side solution is
$D\lambda_s^{-k}$ for $k\le0$, and the terminal linearization gives
\begin{equation}
 D=-\frac{4\lambda_s^{-1}+11}{4\lambda_s+11}=-9+4\sqrt5.                  \label{eq:D-odd}
\end{equation}
The affine constraint gives
\begin{equation}
 E_1=\frac{11-5\sqrt5}{2}.                                                 \label{eq:E1-odd}
\end{equation}
Thus the unique normalized odd profile is
\begin{equation}
 E^{\mathrm o}_k=
 \begin{cases}
 D\lambda_s^{-k},&k\le0,\\
 (11-5\sqrt5)/2,&k=1,\\
 -\lambda_s^k,&k\ge2.
 \end{cases}                                                              \label{eq:E-odd}
\end{equation}

The inverse decay in \eqref{eq:finite-Demko} supplies, in either parity, a
single majorant $M\in\ell^2(\mathbb Z)\cap\ell^3(\mathbb Z)$ for the
normalized profiles.  One may take a multiple of $\sigma^{-k}$ for $k\le0$
and of $|\lambda_s|^k$ on the exterior tail.  This makes the next passage
uniform in the moving endpoint.

Because $\J$ is cubic and $w^*$ is an $\ell^2$ critical point,
\begin{equation}
 \J(w^*+e)-\J(w^*)
 =\frac12D^2\J(w^*)[e,e]+\frac16D^3\J[e,e,e].                             \label{eq:J-Taylor}
\end{equation}
The exact local quadratic density is
\begin{equation}
 \frac12D^2\psi(x,y)[p,q]^2=(x+y)(p+q)^2+2xp^2.                           \label{eq:H-local-half}
\end{equation}
After translating to the endpoint, dominated convergence gives
\begin{equation}
 \frac12D^2\J(1)[E,E]
 =6\sum_{k\in\mathbb Z}E_k^2+4\sum_{k\in\mathbb Z}E_kE_{k+1}.           \label{eq:H-regular}
\end{equation}
The pendant contribution moves to $-\infty$ and vanishes.  The cubic
remainder divided by $A_R^2$ is bounded by
$C|A_R|\sum_k(M_k+M_{k+1})^3$ and therefore tends to zero.

Substitution of \eqref{eq:E-even} into \eqref{eq:H-regular} gives
$2\sqrt5$; substitution of \eqref{eq:E-odd} gives
$2(-20+9\sqrt5)$.  Recalling the factor $1/16$ in
Definition~\ref{def:qhat}, we obtain
\begin{align}
 \widehat q_r-q_*
 &=\frac{\sqrt5}{8}c_*^2|\lambda_s|^{2r}(1+o(1))
 &&(r\to\infty,\ r\ \mathrm{even}),                                    \label{eq:rate-even}\\
 \widehat q_r-q_*
 &=\frac{-5+3\sqrt5}{16}c_*^2|\lambda_s|^{2r}(1+o(1))
 &&(r\to\infty,\ r\ \mathrm{odd}).                                     \label{eq:rate-odd}
\end{align}
When the geometric attainment condition \eqref{eq:qgeom-qhat} holds at
every depth, the same formulas hold for $q_r^{\rm geom}$. Without that
extra condition, the variational formulas above already prove
Theorem~\ref{thm:intro-finite-sections}, including
\begin{equation}
 \lim_{\substack{r\to\infty\\r\equiv\epsilon\pmod2}}
 \frac{\widehat q_{r+2}-q_*}{\widehat q_r-q_*}
 =|\lambda_s|^4
 =\left(\frac{3-\sqrt5}{2}\right)^4.                                    \label{eq:ratio-law}
\end{equation}

\section{The exact boundary-layer constant}
\label{sec:qstar}

The preceding section bounds the cost of stopping the chain. We can now turn a finite numerical profile into a bound on the infinite coefficient. The argument needs an energy enclosure and a bound for the omitted tail, both with their signs under control.

The variational definition \eqref{eq:intro-qstar} is already exact.  We now
prove that it equals the coefficient obtained independently from the local
shoelace expansion, and record a rigorous numerical enclosure.

Extend the half-orbit to all integer edge indices by
\begin{equation}
 \widetilde w_{-j-1}=w_j^*.                                                \label{eq:w-reflection-infinite}
\end{equation}
Then
\begin{equation}
 M_*=\sum_{i\in\mathbb Z}(\widetilde w_i-1)
 =2\sum_{j\ge0}(w_j^*-1).                                                  \label{eq:Mstar-full}
\end{equation}
Define
\begin{equation}
 U_0^*=\frac{5-(w_0^*)^2}{16},\qquad
 U_i^*=\frac{5-(\widetilde w_{i-1}+\widetilde w_i)^2}{16}\quad(i\ne0),    \label{eq:Ustar}
\end{equation}
and put
\begin{equation}
 W_i^*=\widetilde w_{i-2}+\widetilde w_{i-1},\qquad
 k_i^*=\frac{W_i^*}{4}(U_i^*+U_{i-2}^*)-\frac{(W_i^*)^3}{48}.              \label{eq:kstar}
\end{equation}
At the regular state $k_i=-5/48$.  Exponential decay makes
\begin{equation}
 K_*:=\sum_{i\in\mathbb Z}\left(k_i^*+\frac5{48}\right)                  \label{eq:Kstar}
\end{equation}
absolutely convergent.  The coefficient read from the area calculation is
\begin{equation}
 q_*^{(A)}=\frac5{24}-K_*-\frac{(w_0^*)^3}{16}-\frac5{16}M_*.              \label{eq:qstar-area-def}
\end{equation}

\begin{proposition}[Telescoping identity]
\label{prop:qstar-telescoping}
The area and variational definitions agree:
\begin{equation}
 q_*^{(A)}=\frac7{48}+\frac1{16}\J(w^*)=q_*.                              \label{eq:qstar-equality}
\end{equation}
\end{proposition}

\begin{proof}
Fix $L\ge2$, take variables $w_0,\ldots,w_L$, set $w_j=1$ for $j>L$, and
reflect by $w_{-j-1}=w_j$.  Define $U_i,W_i,k_i$ by
\eqref{eq:Ustar}--\eqref{eq:kstar}, with $U_0$ the only special radial
coefficient.  Put
\begin{align}
 M_L&=2\sum_{j=0}^L(w_j-1),                                                \label{eq:M-L}\\
 K_L&=\sum_{i=-L-1}^{L+3}\left(k_i+\frac5{48}\right),                     \label{eq:K-L}\\
 q_{A,L}&=\frac5{24}-K_L-\frac{w_0^3}{16}-\frac5{16}M_L,                  \label{eq:qA-L}\\
 q_{J,L}&=\frac7{48}+\frac1{16}\left(b(w_0)+
                  \sum_{j=0}^{L}\psi(w_j,w_{j+1})\right),
 \quad w_{L+1}=1.                                                         \label{eq:qJ-L}
\end{align}
Let
\begin{align}
 B&=3w_0^2+(w_0+w_1)^2-10,                                                 \label{eq:B-residual}\\
 R_j&=(w_{j-1}+w_j)^2+(w_j+w_{j+1})^2+2w_j^2-10
 \quad(1\le j\le L).                                                     \label{eq:R-residual}
\end{align}
A direct finite-range telescoping gives, for every $L$,
\begin{equation}
 \begin{split}
 q_{A,L}-q_{J,L}={}&
 \frac{(w_L-1)(w_L+1)(w_L+3)}{32}+\frac{w_0+w_1}{32}B\\
 &+\sum_{j=1}^{L-1}\frac{w_{j-1}+w_{j+1}}{32}R_j
 +\frac{w_{L-1}+1}{32}R_L.
 \end{split}                                                              \label{eq:TEL}
\end{equation}
For completeness, the identity has an all-$L$ local induction.  The base
$L=2$ is direct.  For the step, write
$a=w_{L-2}$, $b=w_{L-1}$, $c=w_L$ and subtract the specialization $c=1$.
After multiplication by $96/(c-1)$, the three increments are
\begin{align*}
&3a^2+12ab+3ac+3a+24b^2+24bc+36b+14c^2+41c-40,\\
&2(3b^2+3bc+3b+4c^2+7c-20),\\
&3(a^2+4ab+ac+a+6b^2+6bc+10b+2c^2+9c),
\end{align*}
and the first minus the second minus the third is identically zero.

Now substitute directly $(w_0,\ldots,w_L)=(w_0^*,\ldots,w_L^*)$ and set
only the artificial terminal value $w_{L+1}=1$.  Then
$B=R_1=\cdots=R_{L-1}=0$, whereas
\begin{equation}
 R_L=(w_L^*+1)^2-(w_L^*+w_{L+1}^*)^2=O((2/5)^L).                           \label{eq:last-residual}
\end{equation}
The explicit boundary polynomial in \eqref{eq:TEL} is also $O((2/5)^L)$.
Absolute convergence of $M_L$, $K_L$, and the action follows from the stable
tail.  Passing to the limit in \eqref{eq:TEL} proves
\eqref{eq:qstar-equality}.  No convergence assertion about a different
finite minimization problem is used.
\end{proof}

\begin{proposition}[Certified enclosure]
\label{prop:qstar-enclosure}
The interval in \eqref{eq:intro-qstar-box} contains $q_*$.
\end{proposition}

\begin{proof}
A high-precision stationary truncation is rounded to a rational vector
$\widehat w$ and extended by an exact tail of ones.  Every domain assertion
is checked in rational arithmetic, including
\begin{equation}
 \frac12<\widehat w_i,\qquad 2\widehat w_i^2<5,                            \label{eq:exact-domain-certificate}
\end{equation}
the latter being exactly the upper condition
$\widehat w_i<\sqrt{5/2}$.  The infinite action and every nonzero component
of its gradient, including the first tail component, are rational numbers.
By \eqref{eq:J-coercivity},
\begin{equation}
 \J(\widehat w)-\frac23\|\nabla\J(\widehat w)\|_2^2
 \le\J(w^*)\le\J(\widehat w).                                            \label{eq:certificate-bound}
\end{equation}
Exact evaluation of these two rational endpoints and use of
\eqref{eq:intro-qstar} gives \eqref{eq:intro-qstar-box}; its width is less
than $9.39\times10^{-29}$.
\end{proof}

\section{Conditional asymptotics of the maximal area}
\label{sec:area}

Assume Hypothesis~\ref{thm:imported-even}. We insert the shadowed profile into the exact reconstruction formula
\eqref{eq:exact-reconstruction}.  The indexing is important because there is
one exceptional marked radial coefficient.

Reflection fixes radial vertex zero and sends edge $j$ to $-j-1$; thus the
full finite edge list is
\[
 w_0,w_1,\ldots,w_{m-2},w_{m-1},w_{m-2},\ldots,w_1,w_0.
\]
The opposite edge $m-1$ is fixed.  If $\phi_i$ is the polar angle of $z_i$,
the reconstruction traverses the skeleton by the step $i\mapsto i-2$ and
\begin{equation}
 \phi_{i-2}-\phi_i=-2\pi+(\eta_{i-2}+\eta_{i-1}).                          \label{eq:shoelace-angle}
\end{equation}
The sine is therefore positive in the localized regime.  Hence
\begin{equation}
 A_S=\frac12\sum_{i\in\mathbb Z/N\mathbb Z}
 r_i r_{i-2}\sin(\eta_{i-2}+\eta_{i-1})                                  \label{eq:shoelace}
\end{equation}
and
\begin{equation}
 A_n=A_S+p(\theta),\qquad p(\theta)=\sin\theta(1-\cos\theta).             \label{eq:area-pendant}
\end{equation}

For the finite optimizer define
\begin{equation}
 W_{i,n}=w_{i-2}^{(n)}+w_{i-1}^{(n)},                                     \label{eq:finite-W}
\end{equation}
using the finite cyclic reflection.  In the linear term below this finite
quantity must be retained.  Only the cubic density is frozen to the limiting
orbit.  From \eqref{eq:u-scaled}, \eqref{eq:Ustar}, and the analytic local
expansion,
\begin{equation}
 \frac12r_{i,n}r_{i-2,n}\sin(gW_{i,n})
 =\frac{gW_{i,n}}8+g^3k_i^*+R_{i,n},                                     \label{eq:corrected-local-area}
\end{equation}
where the shadowing estimate and exponential stable tail give
\begin{equation}
 \sum_i|R_{i,n}|
 \le CNg^5+Cg^3(g^2+2^{-m})+Cg^3(2/5)^m=O(N^{-4}).                        \label{eq:area-remainder-sum}
\end{equation}
The first sum is now exactly normalized by turning:
\begin{equation}
 \frac g8\sum_iW_{i,n}=\frac g4\sum_iw_i^{(n)}=\frac\pi4.                \label{eq:linear-area-exact}
\end{equation}
On the other hand, only the cubic density is compared to its half-line
limit:
\begin{equation}
 \sum_i k_i^*=-\frac{5N}{48}+K_*+O((2/5)^m).                              \label{eq:cubic-density-sum}
\end{equation}
Finally \eqref{eq:skeleton-angle-expansion} gives
\begin{equation}
 \theta=\frac g2w_{0,n}+O(g^3)
       =\frac g2w_0^*+O(g^3),\qquad
 p(\theta)=\frac{g^3(w_0^*)^3}{16}+O(g^5).                               \label{eq:pendant-cubic}
\end{equation}
Combining \eqref{eq:area-pendant}--\eqref{eq:pendant-cubic},
\begin{equation}
 A_n=\frac\pi4+g^3\left[-\frac{5N}{48}+K_*+
                  \frac{(w_0^*)^3}{16}\right]+O(N^{-4}).                 \label{eq:area-in-g}
\end{equation}

Use \eqref{eq:g-turning} and put
$B_*=K_*+(w_0^*)^3/16$.  Then
\begin{equation}
 A_n=\frac\pi4+\pi^3\frac{-5N/48+B_*}{(N+M_*)^3}+O(N^{-4}).              \label{eq:area-in-N}
\end{equation}
Since $N=n-1$,
\[
 N+M_*=n+(M_*-1),\qquad
 (N+M_*)^{-3}=n^{-3}\left(1-\frac{3(M_*-1)}n+O(n^{-2})\right).
\]
Collecting the $n^{-3}$ terms gives
\begin{equation}
 \begin{split}
 A_n={}&\frac\pi4-\frac{5\pi^3}{48n^2}\\
 &+\frac{\pi^3}{n^3}
 \left[K_*+\frac{(w_0^*)^3}{16}+\frac{5M_*}{16}-\frac5{24}\right]
 +O(n^{-4}).
 \end{split}                                                              \label{eq:area-expanded}
\end{equation}
By Proposition~\ref{prop:qstar-telescoping}, the bracket equals $-q_*$.  This
proves Theorem~\ref{thm:intro-area}.

We emphasize the logical status.  Localization was applied to the maximizer
from Hypothesis~\ref{thm:imported-even}; the fixed-$C$ uniqueness argument after
\eqref{eq:w-shadow} identified its radial data with the perturbative branch;
and \eqref{eq:exact-reconstruction} is the exact area of that same polygon.
Under that hypothesis, \eqref{eq:intro-area} therefore concerns $A_n$.
Without the hypothesis, the half-line analysis does not by itself prove
global optimality among small polygons.

The error accounting is uniform: local Taylor terms contribute
$N\,O(g^5)=O(n^{-4})$; the profile continuation contributes $O(g^5)$; the
opposite boundary contributes $O(g^3 2^{-m})$; the turning error changes the
leading $Ng^3$ term by $O(n^{-4})$; and the pendant Taylor remainder is
$O(g^5)$.  No discarded term is larger than $O(n^{-4})$.

\section{The Foster--Szab\'o gap}
\label{sec:gap}

For even $n$, Foster and Szab\'o's diameter-graph argument gives the upper
bound~\cite{FosterSzabo2007}
\begin{equation}
 \overline A_n=\frac n2\sin\frac\pi n
 -\frac{n-1}{2}\tan\frac{\pi}{2n-2}.                                     \label{eq:FS-bound}
\end{equation}
Taylor expansion, including the conversion $n-1$ in the second term, gives
\begin{equation}
 \overline A_n=\frac\pi4-\frac{5\pi^3}{48n^2}
 -\frac{\pi^3}{24n^3}+O(n^{-4}).                                         \label{eq:FS-expansion}
\end{equation}
Under Hypothesis~\ref{thm:imported-even}, subtracting
\eqref{eq:intro-area} proves
\begin{equation}
 \boxed{
 \overline A_n-A_n=\left(q_*-\frac1{24}\right)\frac{\pi^3}{n^3}
 +O(n^{-4}).}                                                             \label{eq:FS-gap}
\end{equation}
The rational enclosure proves $q_*>1/24$. Numerically,
\begin{equation}
 q_*-\frac1{24}\approx0.07338831685665818078.\label{eq:gap-certificate}
\end{equation}
In particular the leading gap is strictly positive.  This proves the second
part of Theorem~\ref{thm:intro-area} and gives an conditional extremal interpretation of the variational constant.

\section{The analytic master expansion}
\label{sec:master-expansion}

Assume Hypothesis~\ref{thm:imported-even}. We strengthen the finite-order calculation to an analytic interpolation in
$1/n$.  The argument uses two nearby spatial weights so that the remote
boundary is exponentially small uniformly along the analytic branch.

Fix
\begin{equation}
 \omega_-:=\frac{31}{40},\qquad \omega_+:=\frac45,\qquad
 X_\pm=\ell^1_{\omega_\pm}(\N_0),\qquad
 \|x\|_{X_\pm}=\sum_{j\ge0}\omega_\pm^{-j}|x_j|.                          \label{eq:two-weights}
\end{equation}
Then $\sigma_0<3/4<\omega_-<\omega_+<1$ and
$\omega_-/\omega_+=31/32$.

Let $L=D_a\mathcal G(0,a^*)$ be the half-line Jacobi operator associated
with \eqref{eq:G-zero}.  The finite inverse estimate
\eqref{eq:inverse-decay} gives
\begin{equation}
 \|L_m^{-1}f\|_{1,\omega_-}^{(m)}\le B_-\|f\|_{1,\omega_-}^{(m)}           \label{eq:Lm-minus}
\end{equation}
uniformly in $m$.

\begin{lemma}[Half-line Jacobi isomorphism]
\label{lem:halfline-isomorphism}
The operator $L:X_-\to X_-$ is a bounded isomorphism, and
$\|L^{-1}\|\le B_-$.
\end{lemma}

\begin{proof}
For finitely supported $f$, solve $L_mx^{(m)}=f|_{0:m-1}$ and extend by
zero.  Estimate \eqref{eq:Lm-minus} gives coordinatewise compactness; every
limit solves $Lx=f$ and obeys the same norm bound.  The limit is unique: if
$Lx=0$ with $x\in X_-\subset\ell^2$, monotone passage in the finite coercive
forms gives $-\langle x,Lx\rangle\ge c\|x\|_2^2$, hence $x=0$.  Density of
finitely supported sequences completes the extension to $X_-$.
\end{proof}

Put $s=g^2$.  Filling the removable singularities in
\eqref{eq:Pg}--\eqref{eq:Bg} produces one local analytic bulk function
$\mathcal F(s;x,y,z)$ and one analytic marked function.  The exact
homogeneous state gives the essential cancellation
\begin{equation}
 \mathcal F(s;0,0,0)=0.                                                    \label{eq:analytic-bulk-zero}
\end{equation}
Uniform Cauchy estimates on a common polydisc and the finite bandwidth show
that
\begin{equation}
 \mathcal G:(s,a)\longmapsto\mathcal G(s,a)                               \label{eq:analytic-G}
\end{equation}
is real analytic from a neighborhood of $(0,a^*)$ in $\R\times X_-$ to
$X_-$.  Lemma~\ref{lem:halfline-isomorphism} and the analytic implicit
function theorem~\cite{Deimling1985} give a unique analytic branch
\begin{equation}
 a(s)=a^*+\sum_{k\ge1}s^ka^{[2k]}\in X_-,\qquad \mathcal G(s,a(s))=0.      \label{eq:analytic-a}
\end{equation}
At each order,
\begin{equation}
 La^{[2k]}=F_k(a^{[2]},\ldots,a^{[2k-2]}),                                \label{eq:profile-recursion}
\end{equation}
so the same inverse $L^{-1}$ recursively computes every coefficient.

The exact local coordinate map gives an analytic edge profile $w(s)$, and
\eqref{eq:analytic-bulk-zero} implies that
$(w_j(s)-1)_{j\ge0}$ is analytic with values in $X_-$.  Therefore
\begin{equation}
 M(s):=2\sum_{j\ge0}(w_j(s)-1)                                            \label{eq:M-s}
\end{equation}
is analytic, with $M(0)=M_*$.

Compare the finite solution with the prefix of the exact half-line branch,
not merely with $a^*$.  All rows except the reflected terminal row are then
satisfied exactly.  The $X_-$ bound gives a terminal mismatch
$O(\omega_-^m)$ in the ordinary norm.  Measured in $X_+$, this is
\begin{equation}
 O\left((\omega_-/\omega_+)^m\right)=O((31/32)^m).                         \label{eq:terminal-two-weight}
\end{equation}
The uniform $X_+$ inverse and contraction estimate yield, for some
$\theta\in(0,1)$ and uniformly for $|s|\le s_1$,
\begin{equation}
 \|a^{(m)}(s)-a(s)|_{0:m-1}\|_{X_+}^{(m)}\le C\theta^m.                   \label{eq:analytic-shadow}
\end{equation}
The same estimate holds for the edge variables and localized sums.

The infinite-boundary turning equation is
\begin{equation}
 g[N+M(g^2)]=\pi.                                                          \label{eq:turning-infinite}
\end{equation}
Put $z=N^{-1}$ and $g=zh$.  This becomes
\begin{equation}
 h[1+zM(z^2h^2)]-\pi=0.                                                    \label{eq:h-equation}
\end{equation}
Its derivative in $h$ at $(z,h)=(0,\pi)$ is one.  Hence there is an analytic
$h(z)$ with $h(0)=\pi$ and
\begin{equation}
 \widehat g_N=zh(z).                                                       \label{eq:ghat}
\end{equation}
The exact finite turning equation differs by $O(\theta^m)$; its derivative
is $N+O(1)$, so
\begin{equation}
 g_N-\widehat g_N=O(N^{-2}\theta^m).                                      \label{eq:g-exponential}
\end{equation}

The following lemma makes the analytic area mechanism explicit.

\begin{lemma}[Analytic area defect]
\label{lem:analytic-area-defect}
There are real-analytic functions $B(g)$ and $D(g)$ near zero, both
$O(g^3)$, such that the area of the infinite-boundary profile is
\begin{equation}
 \widehat A_N=\frac\pi4+N B(\widehat g_N)+D(\widehat g_N).                 \label{eq:analytic-area-decomposition}
\end{equation}
The true finite area satisfies
\begin{equation}
 A_n-\widehat A_N=O(\theta_1^m)                                           \label{eq:area-exponential}
\end{equation}
for some $\theta_1\in(0,1)$.
\end{lemma}

\begin{proof}
Let $\mathcal L(g;\text{local profile})$ be the exact local shoelace density
in \eqref{eq:shoelace}.  On the homogeneous state it is analytic and has
the form
\begin{equation}
 \mathcal L_{\mathrm{reg}}(g)=\frac g4+B(g),\qquad B(g)=O(g^3),            \label{eq:B-density}
\end{equation}
the absence of a quadratic term following from oddness of the local sine
factor.  On the half-line branch, subtract the homogeneous density before
summing the two reflected tails.  The local analytic difference contains at
least one factor of the profile defect $w(s)-1\in X_-$, so the series and all
of its parameter derivatives converge normally.  Add the marked pendant
term and use exact turning to remove the sum of the linear $g/4$ terms.  The
resulting localized correction $D(g)$ is analytic.  The local expansion
\eqref{eq:corrected-local-area} and the fact that the marked term begins at
cubic order show $D(g)=O(g^3)$.

This proves \eqref{eq:analytic-area-decomposition}.  Estimate
\eqref{eq:analytic-shadow} controls each absolutely localized difference by
$O(\theta^m)$; the regular bulk part changes by a polynomial factor times
\eqref{eq:g-exponential}.  Absorbing that factor by increasing the base of
the exponential proves \eqref{eq:area-exponential}.
\end{proof}

Substitute \eqref{eq:ghat} into
\eqref{eq:analytic-area-decomposition}.  Because $B(g)=O(g^3)$,
$z^{-1}B(zh(z))$ is analytic at $z=0$ and begins at order $z^2$; the same is
true of the localized term at the required order.  Therefore
\begin{equation}
 \widehat A_N=\frac\pi4+\sum_{k\ge2}b_kN^{-k}                              \label{eq:N-series}
\end{equation}
with a convergent series near $N^{-1}=0$.  Finally
\begin{equation}
 \frac1N=\frac{1/n}{1-1/n}                                                \label{eq:N-to-n-analytic}
\end{equation}
is analytic at $1/n=0$.  Composing \eqref{eq:N-series} and using
\eqref{eq:area-exponential} proves Theorem~\ref{thm:intro-master}.  The first
two coefficients agree with the separately established sharp expansion:
\begin{equation}
 a_2=-\frac{5\pi^3}{48},\qquad a_3=-q_*\pi^3.                             \label{eq:first-master-coefficients}
\end{equation}

\section{Reproducibility}
\label{sec:reproducibility}

The accompanying archive contains standard Python/SymPy scripts for the
finite polynomial identities, recurrence expansions, reflection and
shoelace indexing, the finite-section endpoint algebra, the Foster--Szab\'o
series, and the rational certificate for $q_*$.  In particular, the
certificate evaluates rational endpoints rather than inferring an interval
from floating-point residuals.  The scripts also check the exact domain
condition $2\widehat w_i^2<5$ coordinate by coordinate.

The proofs above do not treat a successful script run as a substitute for
the analytic arguments.  The symbolic calculations verify finite algebraic
identities used in Proposition~\ref{prop:all-depth-cubic},
Proposition~\ref{prop:qstar-telescoping}, and the local area expansion; the
stable-orbit, uniform inverse, localization, and analytic implicit-function
arguments are mathematical proofs independent of floating-point output.

No proprietary solver or private data are required.  A version-pinned
environment and a one-command verification sequence are supplied with the
supplementary archive.

\section*{Data availability}
No empirical data are used. The accompanying supplementary archive contains
all verification scripts, the rational certificate, and instructions for
reproducing the computations. The numerical values in the figure are
computed from the stated finite-dimensional variational problems.

\section*{Use of computational and AI tools}
SymPy was used to check symbolic identities, mpmath to compute the numerical
orbit, and exact rational arithmetic to verify the enclosure. OpenAI ChatGPT
assisted with proof development, mathematical checks, the verification scripts
and English exposition. OpenAI Codex (GPT-6) assisted in this revision with
the worked boundary example, the explanation of the linearization, the
distinction between the recurrence and its geometric application, and the
rerun of the computational checks. The author takes responsibility for the
text, references and mathematical claims.

\bibliographystyle{amsplain}
\bibliography{references}

\end{document}